\documentclass[letterpaper, 10pt]{IEEEtran}

\usepackage{mathtools,dsfont,amsfonts}
\usepackage{enumitem}
\usepackage{subcaption}
\usepackage[mathscr]{euscript}

\usepackage[amsmath,thmmarks]{ntheorem}

\newtheorem{problem}{Problem}
\theorembodyfont{\upshape}
\newtheorem{assumption}{Assumption}

\usepackage[svgnames]{xcolor}
\usepackage{hyperref}
\hypersetup{
    colorlinks=true,
    linkcolor=blue,
    filecolor=magenta,      
    urlcolor=cyan,
    citecolor=Green,
    }
    
\usepackage{booktabs}
\usepackage{subcaption}
\usepackage{graphicx}
\usepackage[sort,compress]{cite}

\newcommand*\ALPHABET{\mathcal}
\newcommand*\PR{\mathds{P}}
\newcommand*\EXP{\mathds{E}}
\newcommand*\IND{\mathds{1}}
\newcommand*\reals{\mathds{R}}

\newcommand*\BELLMAN{\mathscr{B}}
\newcommand*\MISMATCH{\mathscr{D}}

\newcommand*\tuple{\kappa, w}

\newcommand\F{\mathfrak{F}}
\newcommand\FTV{\F_{\mathrm{TV}}}
\newcommand\dTV{d_{\mathrm{TV}}}
\newcommand\dTVw{d_{\mathrm{TV},w}}
\newcommand\FWas{\F_{\mathrm{Was}}}
\newcommand\dWas{d_{\mathrm{Was}}}
\newcommand\FTVw{\F_{\mathrm{TV},w}}
\newcommand\FWasw{\F_{\mathrm{Was},w}}

\DeclareMathOperator{\Lip}{Lip}
\DeclareMathOperator{\Span}{span}
\DeclareMathOperator{\Osc}{osc}

\newcommand*\scale{(\alpha_1,\alpha_2)}

\DeclareMathOperator{\Tr}{Tr}

\begin{document}
\title{Model approximation in MDPs with unbounded per-step cost}
\author{Berk Bozkurt, Aditya Mahajan, Ashutosh Nayyar, and Yi Ouyang%
\thanks{A preliminary version of this paper was presented at CDC 2023~\cite{Bozkurt2023}.}%
\thanks{Berk Bozkurt and Aditya Mahajan are with the Department of Electrical and Computer Engineering, McGill University, Montreal, QC, Canada. (email: berk.bozkurt@mail.mcgill.ca, aditya.mahajan@mcgill.ca)}%
\thanks{Ashutosh Nayyar is with the Department of Electrical and Computer Engineering, University of Southern California, Los Angeles, CA, USA. (email: ashutoshn@usc.edu)}%
\thanks{Yi Ouyang is with Preferred Networks America, Burlingame, CA, USA (email: ouyangyi@preferred-america.com)}%
\thanks{The work at McGill  was supported by IDEaS grant CFPMN2-30, NSERC grant RGPIN-2021-03511, and  IVADO MSc Excellence Fellowship. The work at USC was supported by NSF grants ECCS 2025732 and ECCS 1750041.}
}
\maketitle

\begin{abstract}
   We consider the problem of designing a control policy for an infinite-horizon discounted cost Markov decision process $\mathcal{M}$ when we only have access to an approximate model $\hat{\mathcal{M}}$. How well does an optimal policy $\hat{\pi}^{\star}$ of the approximate model perform when used in the original model $\mathcal{M}$? We answer this question by bounding a weighted norm of the difference between the value function of $\hat{\pi}^\star $ when used in $\mathcal{M}$ and the optimal value function of $\mathcal{M}$.  We then extend our results and obtain potentially tighter upper bounds by considering affine transformations of the per-step cost. We further provide upper bounds that explicitly depend on the weighted distance between cost functions and weighted distance between transition kernels of the original and approximate models. We present examples to illustrate our results.
\end{abstract}


\section{Introduction}
We consider the problem of model approximation in Markov decision processes (MDPs), i.e., the problem of designing  an optimal controller for an MDP using an approximate model (e.g. designing gait controller of a robot using a simulation model).  Let $\ALPHABET M$ denote the true model of the system and let $\hat{\ALPHABET M}$ denote the approximate model. Suppose we solve the approximate model $\hat{\ALPHABET M}$ to identify a policy $\hat\pi^\star$ which is optimal for $\hat{\ALPHABET M}$. How well does $\hat\pi^\star$ perform in the original model $\ALPHABET M$?

Several variations of this question have been studied in the MDP literature. Perhaps the earliest work investigating this is that of Fox~\cite{fox1971finite}, who investigated approximating MDPs by a finite state approximation. In a series of papers, Whitt generalized these results to approximating general MDPs via state aggregation~\cite{whitt1978,whitt1979,whitt1980}. Similar results for state discretization were obtained in~\cite{bertsekas1975convergence,chow1991optimal}, state and action discretization in~\cite{dufour2012approximation} and for models with state dependent discounting in~\cite{haurie1986approximation}. A general framework to view model approximation using the lens of integral probability metrics was presented by M\"uller~\cite{muller1997}.  There have been considerable  advances on these ideas in recent years \cite{saldi2014asymptotic,saldi2017asymptotic,saldi2018finite}, including generalizations to partially observed models~\cite{kara2022near,AIS}. However, these approximation results are restricted to models with bounded per-step cost. 

A related question is that of continuity of optimal policy in model approximation. In particular, if $\{ \hat{\ALPHABET M_n} \}_{n\geq1}$ is a sequence of models that converge to $\ALPHABET M$ in some sense, do the corresponding optimal policies $\{\hat\pi_{n}^\star \}_{n\geq1}$, where $\hat\pi_n^{\star}$ is optimal for $\hat{\ALPHABET M_n}$, converge to an optimal policy for $\ALPHABET M$? One of the earliest work in this direction is that of Fox~\cite{fox1973discretizing}, who studied the continuity of state discretization procedures. Sufficient conditions for continuity of value function on model parameters were presented in~\cite{dutta1994parametric}. There are series of recent papers which significantly generalize these results, including characterizing conditions under which the optimal policy is continuous in model parameters~\cite{saldi2014asymptotic,saldi2016near,saldi2017asymptotic,saldi2019asymptotic,kara2019robustness,kara2020robustness,kara2022robustness}.

The question of model approximation is also relevant for learning optimal policies when the system model is unknown. Therefore, several notions related to model approximation have been studied in the reinforcement learning literature including
approximate homeomorphisms~\cite{ravindran2004approximate,van2020plannable},
bisimulation metrics~\cite{ferns2004metrics,ferns2011bisimulation,castro2009equivalence},
state abstraction~\cite{abel2016near},
and approximate latent state models~\cite{franccois2019overfitting,gelada2019deepmdp}.

The basic results of model approximation may be characterized as follows. Let $\ALPHABET M$ and $\hat{\ALPHABET M}$ be two MDP models with the same state space $\ALPHABET S$ and action space $\ALPHABET A$. Let $\hat \pi^\star \colon \ALPHABET S \to \ALPHABET A$ be an optimal policy for model $\hat {\ALPHABET M}$. Let $V^{\hat \pi^\star} \colon \ALPHABET S \to \reals$ denote the performance of policy $\hat \pi^\star$ in model $\ALPHABET M$ and let $V^\star \colon \ALPHABET S \to \reals$ denote the optimal value function of model $\ALPHABET M$. Most of the existing literature on model approximation provides bounds on $\| V^{\hat \pi^\star} - V^\star \|_{\infty} \coloneqq \sup_{s \in \ALPHABET S} \bigl| V^{\hat \pi^\star}(s) - V^\star(s) \bigr|$ in terms of the parameters of the models $\ALPHABET M$ and $\hat {\ALPHABET M}$. 

However, such bounds are not appropriate for models  with non-compact state spaces and unbounded per-step cost. To illustrate this limitation, consider the linear quadratic regulation (LQR) problem in which the objective is to minimize the infinite-horizon expected discounted total cost. Let $\ALPHABET M$ and $\hat {\ALPHABET M}$ be two such LQR models and $\hat \pi^\star$ be the optimal policy of $\hat {\ALPHABET M}$. It is well known that
\[
   V^\star(s) = s^\intercal P s + q
   \quad\text{and}\quad
   V^{\hat \pi^\star}(s) = s^\intercal P^{\hat \pi^\star} s + q^{\hat \pi^\star},
\]
where $s \in \reals^{n_s}$ is the state, $P$ is the solution of an appropriate Riccati equation, $P^{\hat \pi^\star}$ is a solution of an appropriate Lyapunov equation (which depends on the gain of policy $\hat \pi^\star$) and $q$ and $q^{\hat \pi^\star}$ are constants (where $q^{\hat \pi^\star}$ depends on $P^{\hat \pi^\star}$).  See Sec.~\ref{sec:LQR} for exact details.
For this model, and for many models with unbounded per-step cost, $\| V^\star - V^{\hat \pi^\star} \|_\infty = \infty$. Therefore, the approximation bounds on $\|V^\star - V^{\hat \pi^\star}\|_\infty$ provided by the existing literature will also evaluate to $\infty$ and, as a result, do not provide any insights into the quality of the approximation. 

The standard approach to deal with unbounded per-step cost is to use a weighted norm rather than a sup norm~\cite{muller1997,saldi2017asymptotic,hernandez2012discrete,hernandez2012further}.
However, in most of the existing literature a weighted norm is used to establish existence and uniqueness of a dynamic programming solution. As far as we are aware, the only paper which uses the weighted norm for model approximation in models with unbounded per-step cost is~\cite{saldi2017asymptotic}, where the authors establish sufficient conditions under which $\hat V^\star_n \to V^\star$ and $V^{\hat \pi^\star_n} \to V^\star$, where $\hat V^\star_n$ and $\hat \pi^\star_n$ are value function and optimal policy of a discretized model with grid cells of size less than $1/n$. However, they do not establish the approximation error when a specific approximate model is used. 

Our main contributions in this paper are as follows:
\begin{itemize}
    \item We provide upper bounds on the approximation error in terms of the weighted-norm:
    \[
        \| V^{\hat \pi^\star} - V^\star \|_w \coloneqq
        \sup_{s \in \ALPHABET S}
        \frac{ \bigl| V^{\hat \pi^\star}(s) - V^\star(s) \bigr|}{ w(s) },
    \]
    where $w \colon \ALPHABET S \to [1, \infty)$ is a weight function. 
    Our bounds are derived using a new functional, which we call the \emph{Bellman mismatch functional}.

    \item In the literature on the existence of dynamic programming solution for models with unbounded per-step cost, it is assumed that the weight function is such that the dynamics under all policies satisfies a Lyaponov stability-type condition~\cite{muller1997,saldi2017asymptotic,hernandez2012discrete,hernandez2012further}.
     In contrast, we assume that such a stability condition is satisfied for only a few policies, including the optimal policies of the original and the approximate model, and policies obtained by using optimal and approximate value functions as one-step look ahead value function. (See Assumptions~\ref{assm:stability}--\ref{ass:N2} for the precise definition). 

    \item We then extend our results and obtain potentially tighter upper bounds by considering affine transformations of the per-step cost. These transformations allow us to trade off between the mismatches in the dynamics with the mismatches in the per-step cost.
    
    \item We present examples to illustrate that for suitable choices of the weight functions and affine transformations our bounds are tighter than existing sup-norm bounds, even for models with bounded per-step cost. In addition, we revisit the LQR example mentioned previously and show that the weighted-norm approximation bounds provide meaningful approximation guarantees for such unbounded-cost models.    
    
    \item We provide further upper bounds that explicitly depend on the weighted distance between cost functions and weighted distance between transition kernels of the original and approximate models.
    In the special case when $w(s) \equiv 1$, our bounds recover the existing sup-norm bounds~\cite{muller1997,abel2016near,gelada2019deepmdp} 
\end{itemize}

\subsubsection*{Notation}
We use calligraphic letters to denote sets (e.g. $\ALPHABET S$), uppercase letters to denote random variables (e.g. $S$) and lowercase letters to denote their realizations (e.g. $s$). The space of probability measures on a set $\ALPHABET S$ is expressed by $\Delta(\ALPHABET S)$. 
Subscripts indicate time, so $S_t$ denotes a random variable at time~$t$. $S_{1:t}$ is a short hand notation for $(S_1, \dots, S_t)$. 

We use $\reals$ to denote the set of real numbers, $\mathds{Z}_{\ge 0}$ to denote the set of non-negative integers, $\PR(\cdot)$ to denote the probability of an event, $\EXP[\cdot]$ to denote expectation of a random variable, and $\IND\{\cdot\}$ to denote indicator of an event. For functions $v_1, v_2 \colon \ALPHABET S \to \reals$, the notation $v_1 \le v_2$ denotes that $v_1(s) \le v_2(s)$ for all $s \in \ALPHABET S$. 

\section{Preliminaries}
\subsection{Markov decision processes}
A discrete-time infinite-horizon discounted cost Markov decision process (MDP) is a tuple $\ALPHABET M = \langle \ALPHABET S, \ALPHABET A, P, c, \gamma \rangle$ where
\begin{itemize}
  \item $\ALPHABET S$ is the state space, which is assumed to be a Borel
    space. The state at time $t$ is denoted by $S_t \in \ALPHABET S$. 
  \item $\ALPHABET A$ is the action space, which is assumed to be a Borel
    space. The action at time~$t$ is denoted by $A_t \in \ALPHABET
    A$. 
  \item $P \colon \ALPHABET S \times \ALPHABET A \to \Delta(\ALPHABET S)$ is a
    controlled stochastic kernel, which specifies the system dynamics. In
    particular, for any time $t$ and any $s_{1:t} \in \ALPHABET S^t$, $a_{1:t}
    \in \ALPHABET A^t$ and any Borel set $B \subset \ALPHABET S$, we have
    \begin{align*}
      &\PR(S_{t+1} \in B \mid S_{1:t} = s_{1:t}, A_{1:t} = a_{1:t})
      \\
      &\quad =
      \PR(S_{t+1} \in B \mid S_t = s_t, A_t = a_t)
      \eqqcolon
      P(B \mid s_t, a_t).
    \end{align*}
  \item $c \colon \ALPHABET S \times \ALPHABET A \to \reals$ is the
    per-step cost function which is assumed to be measurable. We further assume that the per-step cost is  bounded from below, i.e., there exists a finite constant $c_{\min}$ such that $c(s,a) \ge c_{\min}$ for all $(s,a) \in \ALPHABET S \times \ALPHABET A$. 
  \item $\gamma \in (0,1)$ is the discount factor.
\end{itemize}

A  stochastic kernel $\pi \colon \ALPHABET S \to \Delta(\ALPHABET A)$ is called a
(time-homogeneous) policy. Let $\Pi$ denote the space of all time-homogeneous
(and possibly randomized) policies. The performance of any policy $\pi \in
\Pi$ starting from an initial state $s \in \ALPHABET S$ is given by 
\begin{equation}\label{eq:value}
  V^\pi(s) = \EXP^\pi \biggl[
    \sum_{t=1}^\infty \gamma^{t-1} c(S_t, A_t) \,\biggm|\, S_1 = s 
  \biggr]
\end{equation}
where $\EXP^{\pi}$ denotes the expectation with respect to the probability measure on all
system variables induced by the choice of policy~$\pi$. The function $V^\pi$
is called the \emph{value function} of policy~$\pi$. 

A policy $\pi^\star \in \Pi$ is called an \emph{optimal policy} if 
\begin{equation}\label{eq:opt}
  V^{\pi^\star}(s) \le V^\pi(s), 
  \quad \forall s \in \ALPHABET S, \forall \pi \in \Pi.
\end{equation}
Since we consider Borel state and action spaces with possibly unbounded (from above)
per-step cost function, an optimal policy is not guaranteed to exist.  If an optimal policy exists, its value function is called the \emph{optimal value function}. 
We focus on MDPs for which optimal value function exists and can be obtained via dynamic programming. We formally define this as dynamic programming solvability in the next section.

\subsection{Dynamic programming solvability}
Let $\ALPHABET V$ denote the space of measurable functions from $\ALPHABET
S \to [\frac{c_{\min}}{1-\gamma},\infty) \cup \{ + \infty \}$. 

\begin{definition}[Weighted norm]
    Given a weight function $w \colon \ALPHABET S \to [1, \infty)$, we define the weighted norm $\| \cdot \|_{w}$ on $\ALPHABET V$ as follows: for any $v \in \ALPHABET V$,
    \[
       \| v \|_w = \sup_{s \in \ALPHABET S}
       \frac{ |v(s)| }{ w(s) }.
    \]
        $\ALPHABET V_w = \bigl\{ v \in \ALPHABET V : \|v\|_w < \infty \bigr\}$
        and
        $\ALPHABET P_w = \bigl\{ p \in \Delta(\ALPHABET S) : \int w dp < \infty \bigr\}$. 
\end{definition}
It can be easily verified that $\| \cdot \|_w$ is a norm and that $\ALPHABET V_w$ is  a Banach space.

\begin{remark}\label{rem:sup-norm}
When the weight function $w(s) \equiv 1$, then the weighted norm $\| v \|_w$ is equal to the sup-norm $\| v \|_\infty \coloneqq \sup_{s \in \ALPHABET S} | v(s)|$. In this case, $\ALPHABET V_w$ is the subset of all  bounded functions in $\ALPHABET V$ and $\ALPHABET P_w$ is the set of all probability measures on~$\ALPHABET S$. 
\end{remark}

\begin{definition}[Bellman operators]
Define the following two operators:
\begin{itemize}
  \item For any $\pi \in \Pi$, define the \emph{Bellman operator} 
    $\BELLMAN^\pi \colon \ALPHABET V \to \ALPHABET V$ as follows: for any $v
    \in \ALPHABET V$, 
    \[
        [ \BELLMAN^\pi v](s) =
        c_\pi(s) + \gamma
        \int_{\ALPHABET S} v(s') P_\pi(ds'\mid s),
     \]
   where $c_\pi(s) = \int_{\ALPHABET A} c(s,a) \pi(da | s),$ and 
   $P_\pi(ds' | s) = \int_{\ALPHABET A} \pi(da | s) P(ds' | s,a)$.

   \item Define the \emph{Bellman optimality operator} 
   $\BELLMAN^\star$ as follows: for any $v \in \ALPHABET V$,
    \[
        [ \BELLMAN^\star v](s) =
        \inf_{a \in \ALPHABET A} \biggl\{c(s,a) + \gamma
        \int_{\ALPHABET S} v(s') P(ds'\mid s,a) \biggr\}.
     \]
\end{itemize}
\end{definition}
\begin{definition}[One-step greedy policy]
We say that a policy $\pi \in \Pi$ is \emph{one-step greedy} with respect to a value function $v \in \ALPHABET V$ if $\BELLMAN^{\pi} v = \BELLMAN^{\star} v$. We denote by $\ALPHABET G(v)$ the set of all one-step greedy policies with respect to $v$.
\end{definition}

\begin{definition}[Dynamic programming solvability]
  An MDP $\ALPHABET M$ is said to be \emph{dynamic programming solvable}
  (DP-solvable, for short) with respect to a weight function $w \colon \ALPHABET S \to [1, \infty)$ if the following conditions are satisfied:
  \begin{enumerate}
    \item 
    For any $v \in \ALPHABET V_w$, $\BELLMAN^\star v \in \ALPHABET V_w$.
   \item There exists a $V^\star \in \ALPHABET V_w$ such that for all $\pi \in \Pi$,
   \[
       V^\star(s) \le V^{\pi}(s),
       \quad \forall s \in \ALPHABET S
   \]
   with equality at all states for at least one $\pi \in \Pi$.
  \item  $V^\star$ is a fixed point of $\BELLMAN^\star$, i.e., it satisfies the dynamic programming equation
      \[
        V = \BELLMAN^\star V.
       \]
  \end{enumerate}
\end{definition}

MDPs with finite state and action spaces are always DP-solvable. For MDPs with general state and action spaces, there are several conditions in the literature which imply DP-solvability. See~\cite{hernandez2012discrete} for an overview.

\subsection{Weighted-norm stability}
\begin{definition}[$(\tuple)$ stability of a policy]\label{def:stability}
   Given an MDP $\ALPHABET M$ and a tuple $(\tuple)$, where 
   $\kappa$ is a positive constant with $\gamma \kappa < 1$ and $w$ is a function from $\ALPHABET S$ to $[1, \infty)$, we say
   a policy $\pi \in \Pi$ is $(\tuple)$ stable if
   \begin{equation}\label{eq:c-pi}
        \| c_\pi \|_w < \infty,
   \end{equation}
   where $c_\pi(s) = \int_{\ALPHABET A} c(s,a) \pi(da | s),$ and $\forall s \in \ALPHABET S$,
   \begin{equation}
      \int_{\ALPHABET S} w(s') P_\pi(ds' \mid s) 
      \le 
      \kappa w(s), 
      \label{eq:kappa}
   \end{equation}
   where $P_\pi(ds' | s) = \int_{\ALPHABET A} \pi(da | s) P(ds' | s,a)$.
\end{definition}
Let $\Pi_S(\tuple)$ denote the set of all $(\tuple)$-stable policies for MDP~$\ALPHABET M$. 
Note that depending on the choice of $(\tuple)$, the set $\Pi_S(\tuple)$ might be empty. 
\begin{remark}
As stated in Remark~\ref{rem:sup-norm}, when $w(s) \equiv 1$, the weighted norm is the same as the sup-norm. Further, with $w(s) \equiv 1$,   inequality \eqref{eq:kappa} of Definition \ref{def:stability} holds with $\kappa =1$ for any policy $\pi$. Thus, for the case of $w(s) \equiv 1$, $\Pi_S(1, w \equiv 1 )$ is the set of all policies $\pi$ for which $\| c_\pi \|_{\infty} < \infty$.
\end{remark}

For our model approximation results developed later, we will assume that certain policies are $(\tuple)$ stable. It is worthwhile to contrast the $(\tuple)$-stability of a policy with a stronger assumption that is typically imposed in the literature~\cite{muller1997,saldi2017asymptotic,hernandez2012discrete,hernandez2012further}. To make that comparison, we define the following (which is the same as \cite[Assumption~8.3.2]{hernandez2012further}):
\begin{definition}[$(\bar \kappa, \bar w)$ stability of the model]\label{def:model-stability}
   Given an MDP $\ALPHABET M$ and a tuple $(\bar \kappa, \bar w)$, where 
   $\bar \kappa$ is a positive constant with $\gamma \bar \kappa < 1$ and $\bar w$ is a function from $\ALPHABET S$ to $[1, \infty)$, we say that  $\ALPHABET M$ is $(\bar \kappa, \bar w)$ stable if there exists a $c_{\max} < \infty$ such that
   \begin{equation}\label{eq:c-pi-bar}
        \| c(\cdot, a) \|_{\bar w} \le c_{\max}, \quad \forall a \in \ALPHABET A
   \end{equation}
   and
   \begin{equation}
    \int_{\ALPHABET S} \bar w(s') P(ds' \mid s,a) 
    \le 
    \bar\kappa \bar w(s),
    \quad \forall s \in \ALPHABET S, \forall a \in \ALPHABET A.
    \label{eq:stab_act}  
   \end{equation}
\end{definition}

\begin{remark}\label{rem:pi-better}
It is shown in~\cite{hernandez2012further} that $(\bar \kappa, \bar w)$ stability of the model is sufficient for DP-solvability. The notion of $(\tuple)$ stability of a policy is weaker. In particular, $(\bar \kappa, \bar w)$ stability of the model implies that any (time-homogeneous) policy is also $(\bar \kappa, \bar w)$ stable. However, $( \kappa,  w)$ stability of a particular policy does not imply $( \kappa,  w)$ stability of the model.  For a given weight function, the smallest value of $\kappa$ that satisfies~\eqref{eq:kappa} is given by
\begin{equation}\label{eq:kappa-bound}
    \kappa_w = \sup_{s \in \ALPHABET S} \frac{\int_{\ALPHABET S} w(s') P_\pi(ds'|s) }{w(s)}
\end{equation}
while the smallest value of $\bar \kappa$ that satisfies equation~\eqref{eq:stab_act} is given by
\begin{equation}\label{eq:bar-kappa-bound}
    \bar \kappa_w = \sup_{s \in \ALPHABET S, a \in \ALPHABET A} \frac{\int_{\ALPHABET S} w(s') P(ds'|s,a) }{w(s)}.
\end{equation}
It is clear from the definitions that $\kappa_w \le \bar \kappa_w$.
We show via an example in Sec.~\ref{sec:pi_better} that using the weaker notion of $(\tuple)$ stability of a policy drastically increases the range of  possible choices of the weight function and leads to tighter approximation bounds.  
\end{remark}

\begin{lemma}\label{lem:bounded}
    Given an MDP $\ALPHABET M$ and a tuple $(\tuple)$, 
    for any policy $\pi \in \Pi_S(\tuple)$, we have the following:
    \begin{enumerate}
        \item If $v \in \ALPHABET V_w$, then $\BELLMAN^\pi v \in \ALPHABET V_w$.
        \item $\BELLMAN^\pi$ is a $\|\cdot\|_w$-norm contraction with contraction factor $\gamma \kappa$, i.e., for any  $v_1, v_2 \in  \ALPHABET V_w$, 
        we have
      \[
         \| \BELLMAN^\pi v_1 - \BELLMAN^\pi v_2 \|_w 
         \le
         \gamma \kappa \| v_1 - v_2 \|_w.
      \]
      \item The fixed point equation
      \[
           V = \BELLMAN^\pi V
      \]
      has a unique solution in $\ALPHABET V_w$ and that solution is equal to $V^\pi$.
    \end{enumerate}
\end{lemma}
See Appendix~\ref{app:bounded} for proof.

\section{Problem formulation and approximation bounds}\label{sec:main_results}
\subsection{Model approximation in MDPs}
We are interested in the problem of model approximation in MDPs. In particular, suppose there is an MDP $\ALPHABET M = \langle \ALPHABET S, \ALPHABET A, P, c, \gamma \rangle$ of interest, but the system designer has access to only an approximate model $\hat {\ALPHABET M} = \langle \ALPHABET S, \ALPHABET A, \hat P, \hat c, \gamma \rangle$. Note that both models $\ALPHABET M$ and $\hat {\ALPHABET M}$ have the same state and action spaces, but have different transition dynamics and per-step cost. As before, we assume that both $c$ and $\hat c$ are  bounded from below. Thus, there exists a finite constant $c_{\min}$ such that $c(s,a) \ge c_{\min}$ and $\hat c(s,a) \ge c_{\min}$ for all $(s,a) \in \ALPHABET S \times \ALPHABET A$. 

We further assume that both models $\ALPHABET M$ and $\hat {\ALPHABET M}$ are well-behaved in the following sense, which we assume to hold in the rest of the paper.
\begin{assumption}\label{assm:DP-solvability}
    Models $\ALPHABET M$ and $\hat {\ALPHABET M}$ are DP-solvable. 
\end{assumption}

\begin{table}[!t]
    \centering
    \caption{Notation for the variables used for the two models}
    \label{tab:notation}
    \begin{tabular}{@{}lcc@{}}
        \toprule
        Variable & Model $\ALPHABET M$ & Model $\hat {\ALPHABET M}$ \\
        \midrule
        Dynamics & $P$ & $\hat P$ \\
        per-step cost & $c$ & $\hat c$ \\
        Value function of policy $\pi$ & $V^\pi$ & $\hat V^{\pi}$ \\
        Optimal value function & $V^\star$ & $\hat V^\star$ \\
        Optimal policy & $\pi^\star$ & $\hat \pi^\star$ \\
        Bellman operator of policy $\pi$ & $\BELLMAN^\pi$ & $\hat \BELLMAN^\pi$ \\
        Bellman optimality operator & $\BELLMAN^\star$ & $\hat \BELLMAN^\star$ \\
        Set of one-step greedy policies w.r.t.\ $v$ & $\ALPHABET G(v)$ & $\hat {\ALPHABET G}(v)$ \\
        Set of $(\tuple)$-stable policies & $\Pi_S(\tuple)$ & $\hat \Pi_S(\tuple)$ \\
        \bottomrule
    \end{tabular}
\end{table}

We will use the superscript $\hat{}$ (hat) to denote variables/operators corresponding to the approximate model, as summarized in Table~\ref{tab:notation}. We are interested in the following approximation problem.

\begin{problem}\label{prob:main}
  Let $\hat \pi^\star$ be an optimal policy for the approximate model $\hat {\ALPHABET M}$. For each start state $s$, provide a bound for the loss in performance when using $\hat \pi^\star$ in the original model $\ALPHABET M$, i.e., bound $V^{\hat \pi^\star}(s) - V^\star(s)$.
\end{problem}

\subsection{Approximation bounds}

In the rest of the paper, we will work with a fixed $(\tuple)$  where $\kappa$ is a non-negative constant such that $\gamma \kappa < 1$ and $w \colon \ALPHABET S \to [1, \infty)$.  
Note that $\Pi_S(\tuple)$ and $\hat \Pi_S(\tuple)$ denote the sets of $(\tuple)$-stable policies for models $\ALPHABET M$ and $\hat {\ALPHABET M}$, respectively. Also, $\ALPHABET G(v)$ and $\hat{\ALPHABET G}(v)$ denote the sets of one-step greedy polices with respect to $v$ for models $\ALPHABET M$ and $\hat {\ALPHABET M}$, respectively.
We impose the following additional assumption on the models.

\begin{assumption}\label{assm:stability}
We assume that
\begin{enumerate}
    \item The set $\ALPHABET G(V^\star) \cap \Pi_S(\tuple)$ is nonempty.
    \item The set $\hat {\ALPHABET G}(\hat V^\star)\cap \Pi_S(\tuple) \cap \hat \Pi_S(\tuple)$ is nonempty.
\end{enumerate}
\end{assumption}
When $\ALPHABET G(V^\star) \cap \Pi_S(\tuple) \neq \emptyset$, we can show that any policy $\pi \in \ALPHABET G(V^\star) \cap \Pi_S(\tuple)$ is optimal, i.e., $V^{\pi} = V^\star$. From now on, we assume that the optimal policy $\pi^\star$ for Model $\ALPHABET M$ belongs to $\ALPHABET G(V^\star) \cap \Pi_S(\tuple)$. Similarly, we assume that the optimal policy $\hat \pi^\star$ for Model $\hat {\ALPHABET M}$ belongs to $\hat {\ALPHABET G}(\hat V^\star)\cap \Pi_S(\tuple) \cap \hat \Pi_S(\tuple)$.

For some of the results, we impose one of the following assumptions:
\begin{assumption}\label{ass:N1}
    The set $\ALPHABET G(\hat V^\star) \cap  \Pi_S(\tuple)$ is nonempty.
\end{assumption}
\begin{assumption}\label{ass:N2}
    The set $\hat {\ALPHABET G}(V^\star) \cap  \hat \Pi_S(\tuple)$ is nonempty.
\end{assumption}

Assumption~\ref{ass:N1} effectively states using $\hat V^\star$ as the one-step look ahead value function in the original model produces a stable policy.  Similarly, Assumption~\ref{ass:N2} effectively states that using $V^\star$ as the one-step look ahead value function in the approximate model produces a stable policy. 

\begin{definition}[Bellman mismatch functionals] 
\label{def:bellman_mismatch}
Suppose Assumptions \ref{assm:DP-solvability} and \ref{assm:stability} hold. Define the following functionals:
\begin{itemize}
    \item  For any $\pi \in \Pi_S(\tuple)$ and $\hat\pi \in \hat \Pi_S(\tuple)$, define the \emph{Bellman mismatch functional}  $\MISMATCH_w^{\pi,\hat\pi} \colon \ALPHABET V_w \to \reals_{\ge 0}$  as follows: for any $v \in \ALPHABET V_w$,
    \[
        \MISMATCH^{\pi,\hat\pi}_w v = \| \BELLMAN^{\pi} v - \hat{\BELLMAN}^{\hat \pi} v \|_w.
    \]

    \item For any $\pi \in \Pi_S(\tuple) \cap \hat \Pi_S(\tuple)$, define the \emph{Bellman mismatch functional} $\MISMATCH^\pi_w \colon \ALPHABET V_w \to \reals_{\ge 0}$ as follows: for any $v \in \ALPHABET V_w$, 
    \[
        \MISMATCH_w^{\pi}v = 
        \MISMATCH_w^{\pi,\pi}v
        =
        \| \BELLMAN^{\pi}v - \hat{\BELLMAN}^{\pi}v \|_w.
    \]

    \item Define the \emph{Bellman optimality mismatch functional} $\MISMATCH^\star_w \colon \ALPHABET V_w \to \reals_{\ge 0}$  as follows: for any $v \in \ALPHABET V_w$, 
    \[
        \MISMATCH_w^{\star}v = \| \BELLMAN^{\star} v - \hat{\BELLMAN}^{\star} v \|_w.
    \]
\end{itemize}
\end{definition}
In the rest of the paper, we assume that Assumption~\ref{assm:stability} holds and $w$ is fixed. Therefore, we omit the subscript $w$ in the mismatch functionals  in the rest of the discussion.

\begin{lemma}[Policy error bounds]\label{lem:policy-error}
For any two policies $\pi \in \Pi_S(\tuple)$ and $\hat\pi \in \hat \Pi_S(\tuple)$,  we have
    \begin{equation}\label{eq:DeltaVpi}
        \| V^\pi - \hat V^{\hat\pi} \|_w \le \frac{1}{1-\gamma \kappa} 
        \min \bigl\{ \MISMATCH^{\pi,\hat\pi} V^\pi, \MISMATCH^{\pi,\hat\pi} \hat V^{\hat\pi} \bigr\}.
    \end{equation}
\end{lemma}
See Appendix~\ref{app:policy-error} for proof.

\begin{lemma}[Value error bounds]\label{lem:value-error}
    The following hold:
    \begin{enumerate}
        \item
         If Assumptions~\ref{assm:DP-solvability} and \ref{assm:stability} hold, we have
        \begin{equation}\label{eq:N-1}
            \| V^\star - \hat V^\star \|_w \le \frac 1{(1 - \gamma \kappa)} \MISMATCH^{\pi^\star, \hat \pi^\star} \hat V^\star.
        \end{equation}
        and 
        \begin{equation}\label{eq:N0}
            \| V^\star - \hat V^\star \|_w \le \frac 1{(1 - \gamma \kappa)} \MISMATCH^{\pi^\star, \hat \pi^\star} V^\star.
        \end{equation}
        \item If Assumptions~\ref{assm:DP-solvability}, \ref{assm:stability} and \ref{ass:N1}  hold, we have
        \begin{equation}\label{eq:N1}
            \| V^\star - \hat V^\star \|_w \le \frac 1{(1 - \gamma \kappa)} \MISMATCH^\star \hat V^\star.
        \end{equation}
        \item If Assumptions~\ref{assm:DP-solvability}, \ref{assm:stability} and \ref{ass:N2}  hold, we have
            \begin{equation}\label{eq:N2}
                \| V^\star - \hat V^\star \|_w \le \frac 1{(1 - \gamma \kappa)} \MISMATCH^\star V^\star.
            \end{equation}
    \end{enumerate}
\end{lemma}
See Appendix~\ref{app:value-error} for proof.

We can establish the following theorem by combining policy and value error bounds.
\begin{theorem}\label{thm:bound}
   We have the following bounds on $V^{\hat \pi^\star} - V^\star$:
   \begin{enumerate}
   \item Under Assumptions~\ref{assm:DP-solvability} and \ref{assm:stability}, we have
       \begin{align*}
          \bigl\| V^{\hat \pi^\star} - V^\star \bigr\|_w
          &\leq
          \frac{1}{1 - \gamma \kappa }
          \bigl[ \MISMATCH^{\hat \pi^\star} \hat V^\star
          + \MISMATCH^{\pi^\star, \hat \pi^\star} \hat V^\star \bigr]
       \end{align*}
       and
    \begin{align*}
      \bigl\| V^{\hat \pi^\star} - V^\star \bigr\|_w
      &\le
      \frac{1}{1 - \gamma \kappa }
      \MISMATCH^{\hat \pi^\star} V^\star 
      +
      \frac{(1 + \gamma \kappa)}{(1 - \gamma \kappa)^2 }
      \MISMATCH^{\pi^\star, \hat \pi^\star} V^\star. 
   \end{align*}
    \item Under Assumptions~\ref{assm:DP-solvability}, \ref{assm:stability}, and ~\ref{ass:N1}, we have
       \begin{align*}
          \bigl\| V^{\hat \pi^\star} - V^\star \bigr\|_w
          &\leq
          \frac{1}{1 - \gamma \kappa }
          \bigl[ \MISMATCH^{\hat \pi^\star} \hat V^\star
          + \MISMATCH^\star \hat V^\star \bigr].
       \end{align*}
    \item Under Assumptions~\ref{assm:DP-solvability}, \ref{assm:stability}, and ~\ref{ass:N2}, we have
       \begin{align*}
          \bigl\| V^{\hat \pi^\star} - V^\star \bigr\|_w
          &\le
          \frac{1}{1 - \gamma \kappa }
          \MISMATCH^{\hat \pi^\star} V^\star 
          +
          \frac{(1 + \gamma \kappa)}{(1 - \gamma \kappa)^2 }
          \MISMATCH^{\star} V^\star. 
       \end{align*}
   \end{enumerate}
\end{theorem}
See Appendix~\ref{app:thm:bound} for proof.

\begin{remark}
   Since $  V^{\hat \pi^\star}(s) \ge V^{\star}(s)$, we have 
   \begin{equation}\label{eq:V-bound}
    V^{\hat \pi^\star}(s) - V^\star(s) \le \bigl\| V^{\hat \pi^\star} - V^\star \bigr\|_w w(s).
    \end{equation}
    Thus, the bounds on $ \bigl\| V^{\hat \pi^\star} - V^\star \bigr\|_w$ stated in Theorem~\ref{thm:bound} provide a bound on the performance loss when   $\hat \pi^\star$ is used in the original model $\ALPHABET M$ with a start state $s$. 
\end{remark}

\subsection{Discussion} \label{sec:discussion}
Obtaining a solution of Problem~\ref{prob:main} requires some knowledge of the model. If we were to obtain an exact expression for $\|V^{\star} - V^{\hat \pi^\star} \|_w$, we would need to compute $V^\star$ and $V^{\hat \pi^\star}$, which are the fixed points of $\BELLMAN^\star$ and $\BELLMAN^{\hat \pi^\star}$, respectively. Computing $V^\star$ and $V^{\hat \pi^\star}$ requires starting with an initial choice $V_0$ and then iteratively computing $\{(\BELLMAN^\star)^n V_0\}_{n \ge 1}$ and $\{(\BELLMAN^{\hat \pi^\star})^n V_0\}_{n \ge 1}$ until convergence. In contrast, our upper bounds of Theorem~\ref{thm:bound}, part~2, are in terms of the mismatch Bellman operators, which require \emph{one} update of the Bellman operators $\BELLMAN^{\star}$ and $\BELLMAN^{\hat \pi^\star}$. It is worth highlighting that we do not need to compute $V^\star$ or $\pi^\star$ in order to use the bounds of Theorem~\ref{thm:bound}, part~2.
Thus, our upper bounds provide significant computational savings, especially when computing a Bellman update in the original model is computationally expensive.  

Another feature of our results is that they characterize the sensitivity of the optimal performance to model approximation. To make this notion precise, we need to define a notion of \emph{distance} between the original and approximate model. We elaborate on this direction in Sec.~\ref{sec:IPM}. We first present a few generalizations of the bounds.

\subsection{Bounds under stability of deterministic open loop policies}
Let $\pi_a$ denote the deterministic open loop policy that selects action $a$ with probability $1$ in all states, i.e., $\pi_a(s) =a$ for all $s$. In this section, we assume that all such policies are $(\tuple)$ stable in $ {\ALPHABET M}$ and $\hat {\ALPHABET M}$, and simplify the  bounds of Theorem \ref{thm:bound}.
\begin{assumption}\label{assm:openloop}
For each $a \in \ALPHABET A$,  $\pi_a \in \Pi_S(\tuple) \cap \hat \Pi_S(\tuple)$.
Moreover $\ALPHABET G(\hat V^\star)$ and $\hat {\ALPHABET G}(V^\star)$ are nonempty.
\end{assumption}

Assumption~\ref{assm:openloop} is weaker than $(\kappa, w)$ stability of the model because Assumption~\ref{assm:openloop} does not imply a \emph{uniform} upper bound on $\| c(\cdot,a)\|_w$ over all $a \in \ALPHABET A$.

\begin{lemma}\label{lem:assm_5_4}
    Assumptions~\ref{assm:DP-solvability} and \ref{assm:openloop} imply Assumptions~\ref{ass:N1} and~\ref{ass:N2}.
\end{lemma}
See Appendix~\ref{app:assm_5_4} for proof.

 \begin{definition}\label{def:openloop}
     Suppose Assumption \ref{assm:openloop} holds. Define the \emph{Bellman maximum mismatch functional} $\MISMATCH^{\max}_w \colon \ALPHABET V_w \to \reals_{\ge 0}$  as follows: for any $v \in \ALPHABET V_w$, 
     \begin{align*}
         \MISMATCH_w^{\max}v 
         &= \sup_{a \in \ALPHABET A} \MISMATCH^{\pi_a}_w v
         = \sup_{a \in  \ALPHABET A} \| \BELLMAN^{\pi_a}v - \hat{\BELLMAN}^{\pi_a}v \|_w .
     \end{align*}
\end{definition}
In the sequel, we omit the subscript $w$ from the functional defined above for simplicity.

 \begin{lemma}\label{lem:mismatch}
    Under Assumptions~\ref{assm:DP-solvability}, \ref{assm:stability} and~\ref{assm:openloop}, the  Bellman mismatch functionals satisfy the following for any $v \in \ALPHABET V_w$:
        \begin{equation}\label{eq:mismatch-reln}
            \sup_{\pi \in \Pi_S(\tuple) \cap \hat \Pi_S(\tuple)} \MISMATCH^\pi v = \MISMATCH^{\max} v
            \quad\text{and}\quad
            \MISMATCH^\star v \le \MISMATCH^{\max} v.
        \end{equation}
\end{lemma}
See Appendix~\ref{app:mismatch} for proof.

\begin{theorem}\label{thm:openloop bound}
   Under Assumptions~\ref{assm:DP-solvability}, \ref{assm:stability} and~\ref{assm:openloop}, we have the following two bounds on $V^{\hat \pi^\star} - V^\star$:
   \begin{enumerate}
       \item \textbf{Bound in terms of properties of $\hat V^\star$:}
       \begin{align*}
          \bigl\| V^{\hat \pi^\star} - V^\star \bigr\|_w
             &\leq \frac{2}{1-\gamma\kappa} \MISMATCH^{\max}(\hat V^\star).
       \end{align*}
       \item \textbf{Bound in terms of properties of $V^\star$:}
       \begin{align*}
          \bigl\| V^{\hat \pi^\star} - V^\star \bigr\|_w
             &\leq \frac{2}{(1-\gamma\kappa)^2} \MISMATCH^{\max}(V^\star).
       \end{align*}
   \end{enumerate}
\end{theorem}
\begin{proof}
    The result follows from Theorem \ref{thm:bound} (parts 2 and 3), Lemma~\ref{lem:assm_5_4}, and Lemma~\ref{lem:mismatch}.
\end{proof}

\subsection{Generalized  bounds based on affine transformations of the cost}
\label{sec:alpha_beta}

Given an MDP $\ALPHABET M$ and a tuple $\boldsymbol{\alpha} = \scale$ of real numbers where $\alpha_1 > 0$, define a new MDP $\ALPHABET M_{\boldsymbol{\alpha}}$ with the same dynamics as $\ALPHABET M$ but with the cost function  modified to $\alpha_1  c(s,a) + \alpha_2$. For any policy $\pi$, let $V^\pi_{\boldsymbol{\alpha}}$ denote the value function of $\pi$ in $\ALPHABET M_{\boldsymbol{\alpha}}$. Similarly, let $V^\star_{\boldsymbol{\alpha}}$ denote the optimal value function for $\ALPHABET M_{\boldsymbol{\alpha}}$. 

\begin{lemma}\label{lem:alpha-beta}
   The following properties hold for any $s \in \ALPHABET S$:
   \begin{enumerate}
       \item For any policy $\pi$, 
       $V^\pi_{\boldsymbol{\alpha}}(s) = \alpha_1  V^\pi(s) + {\alpha_2}/{(1-\gamma)}$.
       \item  If $\pi^\star$ is optimal for $\ALPHABET M$, then it is also optimal for $\ALPHABET M_{\boldsymbol{\alpha}}$ and
       $V^\star_{\boldsymbol{\alpha}}(s) = \alpha_1  V^\star(s) + {\alpha_2}/{(1-\gamma)}$.
       \item For any policy $\pi$ and weight function $w \colon  \ALPHABET S \to [1, \infty)$, 
       $\| V^{\pi}_{\boldsymbol{\alpha}} - V^\star_{\boldsymbol{\alpha}}\|_w 
       = \alpha_1 \| V^\pi - V^\star \|_w$.
   \end{enumerate}
\end{lemma}
\begin{proof}
    Properties 1 and 2 are immediate consequences of the definitions. Property 3 follows from properties 1 and~2.
\end{proof}
Lemma~\ref{lem:alpha-beta} provides an alternative way of bounding the performance loss when the optimal policy $\hat \pi^\star$ for the approximate model $\hat {\ALPHABET M}$ is used in the true model ${\ALPHABET M}$. We can first view $\hat {\ALPHABET M}$ as an approximation for $\ALPHABET M_{\boldsymbol{\alpha}}$ and bound the approximation error $\|V^{\hat \pi^\star}_{\boldsymbol{\alpha}} - V^\star_{\boldsymbol{\alpha}}\|_w$ in $\ALPHABET M_{\boldsymbol{\alpha}}$. 
 Part 3 of Lemma~\ref{lem:alpha-beta} implies that the approximation error in $\ALPHABET M$ is   simply $1/\alpha_1$ times the approximation error in $\ALPHABET M_{\boldsymbol{\alpha}}$. 
 
 To bound the approximation error in $\ALPHABET M_{\boldsymbol{\alpha}}$, let $\BELLMAN^\pi_{\boldsymbol{\alpha}}$ and $\BELLMAN^\star_{\boldsymbol{\alpha}}$ denote the Bellman operator and Bellman optimality operator for $\ALPHABET M_{\boldsymbol{\alpha}}$. Let $\ALPHABET G_{\boldsymbol{\alpha}}(v)$ denote the set of one-step greedy policies with respect to $v$ in model $\ALPHABET M_{\boldsymbol{\alpha}}$. 
 
Note that for any $\scale$ with $\alpha_1 > 0$, we have that:
    (i)~If $\ALPHABET M$ is DP-solvable, then so is $\ALPHABET M_{\boldsymbol{\alpha}}$;
  (ii) the set of  $(\tuple)$-stable policies is the same  for $\ALPHABET M$ and $\ALPHABET M_{\boldsymbol{\alpha}}$. Consequently, if any of  Assumptions~\ref{assm:DP-solvability}, \ref{assm:stability} or \ref{assm:openloop} holds for  $\ALPHABET M$ and $\hat {\ALPHABET M}$, then it also holds for $\ALPHABET M_{\boldsymbol{\alpha}}$ and $\hat {\ALPHABET M}$.  Instead of Assumptions~\ref{ass:N1} and~\ref{ass:N2}, we need the following  alternative assumptions.
  
\begin{assumption}\label{ass:N1_alt}
    The set $\ALPHABET G_{\boldsymbol{\alpha}}(\hat V^\star) \cap  \Pi_S(\tuple)$ is nonempty.
\end{assumption}
\begin{assumption}\label{ass:N2_alt}
    The set $\hat {\ALPHABET G}(V_{\boldsymbol{\alpha}}^\star) \cap  \hat \Pi_S(\tuple)$ is nonempty.
\end{assumption}
 
 We can now define mismatch functionals (analogous to those defined in Definitions~\ref{def:bellman_mismatch} and~\ref{def:openloop}) using $\ALPHABET M_{\boldsymbol{\alpha}}$ and $\hat {\ALPHABET M}$.
 
 \begin{definition}
\label{def:bellman_mismatch_alpha}
Suppose Assumptions~\ref{assm:DP-solvability} and~\ref{assm:stability} hold. Define the following functionals:
\begin{itemize}
    \item For any $\pi \in \Pi_S(\tuple)$ and $\hat\pi \in \hat \Pi_S(\tuple)$, define the \emph{Bellman mismatch functional}  $\MISMATCH_{\boldsymbol{\alpha}}^{\pi,\hat\pi} \colon \ALPHABET V_w \to \reals_{\ge 0}$  as follows: for any $v \in \ALPHABET V_w$,
    \[
        \MISMATCH^{\pi,\hat\pi}_{\boldsymbol{\alpha}} v = \| \BELLMAN^{\pi}_{\boldsymbol{\alpha}} v - \hat{\BELLMAN}^{\hat \pi} v \|_w.
    \]

    \item For any $\pi \in \Pi_S(\tuple) \cap \hat \Pi_S(\tuple)$, define the \emph{Bellman mismatch functional} $\MISMATCH^\pi_{\boldsymbol{\alpha}} \colon \ALPHABET V_w \to \reals_{\ge 0}$ as follows: for any $v \in \ALPHABET V_w$ 
    \[
        \MISMATCH^\pi_{\boldsymbol{\alpha}} v = \MISMATCH^{\pi,\pi}_{\boldsymbol{\alpha}}v= \| \BELLMAN^{\pi}_{\boldsymbol{\alpha}}v - \hat{\BELLMAN}^{\pi}v \|_w.
    \]
    
    \item Define the \emph{Bellman optimality mismatch functional} $\MISMATCH^\star_{\boldsymbol{\alpha}} \colon \ALPHABET V_w \to \reals_{\ge 0}$  as follows: for any $v \in \ALPHABET V_w$, 
    \[
         \MISMATCH_{\boldsymbol{\alpha}}^{\star}v = \| \BELLMAN^{\star}_{\boldsymbol{\alpha}} v - \hat{\BELLMAN}^{\star} v \|_w.
    \]
    
\end{itemize}
\end{definition}

\begin{definition}
     Suppose Assumption \ref{assm:openloop} holds. Define the \emph{Bellman maximum mismatch functional} $\MISMATCH^{\max}_{\boldsymbol{\alpha}} \colon \ALPHABET V_w \to \reals_{\ge 0}$  as follows: for any $v \in \ALPHABET V_w$, 
     \[
         \MISMATCH_{\boldsymbol{\alpha}}^{\max}v =\sup_{a \in \ALPHABET A} \MISMATCH_{\boldsymbol{\alpha}}^{\pi_a} v
         =\sup_{a \in  \ALPHABET A} \| \BELLMAN^{\pi_a}_{\boldsymbol{\alpha}}v - \hat{\BELLMAN}^{\pi_a}v \|_w .
     \]
\end{definition}

We can now use Lemma~\ref{lem:alpha-beta} to present variants of  Theorem~\ref{thm:bound} and Theorem~\ref{thm:openloop bound}.

\begin{theorem}\label{thm:alpha_beta}
    For any $\boldsymbol{\alpha} = \scale$ with $\alpha_1 > 0$, we have the following bounds on $V^{\hat \pi^\star} - V^\star$:
   \begin{enumerate}
   \item Under Assumptions~\ref{assm:DP-solvability} and~\ref{assm:stability}, we have
          \begin{align*}
          \bigl\| V^{\hat \pi^\star} - V^\star \bigr\|_w
          &\leq
          \frac{1}{\alpha_1(1 - \gamma \kappa) }
          \bigl[ \MISMATCH_{\boldsymbol{\alpha}}^{\hat \pi^\star} \hat V^\star
          + \MISMATCH_{\boldsymbol{\alpha}}^{\pi^\star, \hat \pi^\star} \hat V^\star \bigr]
       \end{align*}
       and
    \begin{align*}
      \bigl\| V^{\hat \pi^\star} - V^\star \bigr\|_w
      &\le
      \frac{1}{\alpha_1(1 - \gamma \kappa) }
      \MISMATCH_{\boldsymbol{\alpha}}^{\hat \pi^\star} (\alpha_1 V^\star) \notag \\
      &\quad +
      \frac{(1 + \gamma \kappa)}{\alpha_1(1 - \gamma \kappa)^2 }
      \MISMATCH_{\boldsymbol{\alpha}}^{\pi^\star, \hat \pi^\star}(\alpha_1 V^\star). 
   \end{align*}
    \item Under Assumptions~\ref{assm:DP-solvability}, \ref{assm:stability} and~\ref{ass:N1_alt}, we have
      \begin{align*}
          \bigl\| V^{\hat \pi^\star} - V^\star \bigr\|_w
          &\le 
          \frac{ 1  }
         {\alpha_1 (1 - \gamma \kappa) }
          \bigl[ \MISMATCH^{\hat \pi^\star}_{\boldsymbol{\alpha}} \hat V^\star
          + \MISMATCH^{\star}_{\boldsymbol{\alpha}} \hat V^\star \bigr].
       \end{align*}
       \item Under Assumptions~\ref{assm:DP-solvability}, \ref{assm:stability} and~\ref{ass:N2_alt}, we have
    \begin{align*}
        \bigl\| V^{\hat \pi^\star} - V^\star \bigr\|_w
          &\le
          \frac{1}{\alpha_1(1 - \gamma \kappa) }
          \MISMATCH^{\hat \pi^\star}_{\boldsymbol{\alpha}} (\alpha_1 V^\star)
          \notag \\
          &\quad +
          \frac{(1 + \gamma \kappa)}{\alpha_1(1 - \gamma \kappa)^2 }
          \MISMATCH^{\star}_{\boldsymbol{\alpha}} (\alpha_1 V^\star). 
    \end{align*}
   \end{enumerate}
   \end{theorem}
See Appendix~\ref{app:alpha_beta} for the proof.

\begin{theorem}\label{thm:alpha_beta_openloop}
   Under Assumptions~\ref{assm:DP-solvability}, \ref{assm:stability} and~\ref{assm:openloop}, we have the following two bounds on $V^{\hat \pi^\star} - V^\star$:
   \begin{enumerate}
       \item \textbf{Bound in terms of properties of $\hat V^\star$:}
       \begin{align*}
          \bigl\| V^{\hat \pi^\star} - V^\star \bigr\|_w
             &\le \frac{2}{\alpha_1(1-\gamma\kappa)} \MISMATCH_{\boldsymbol{\alpha}}^{\max}\hat V^\star.
       \end{align*}
       \item \textbf{Bound in terms of properties of $V^\star$:}
       \begin{align*}
          \bigl\| V^{\hat \pi^\star} - V^\star \bigr\|_w
              \leq \frac{2}{\alpha_1(1-\gamma\kappa)^2} \MISMATCH_{\boldsymbol{\alpha}}^{\max}(\alpha_1 V^\star).
       \end{align*}
   \end{enumerate}
\end{theorem}

\begin{proof}The result follows from Theorem \ref{thm:alpha_beta} (parts 2 and 3), Lemma \ref{lem:mismatch} and the fact that Assumption \ref{assm:openloop} implies Assumptions~\ref{ass:N1_alt} and \ref{ass:N2_alt} (using the same argument as Lemma~\ref{lem:assm_5_4}).
\end{proof}

\subsubsection*{Some remarks}
\begin{itemize}
    \item It is possible to optimize the bounds in Theorems~\ref{thm:alpha_beta} and~\ref{thm:alpha_beta_openloop} by optimizing over multiple choices of $\boldsymbol{\alpha}$. For Theorem~\ref{thm:alpha_beta}, parts~2 and~3, we need to ensure that the choice of $\alpha$ satisfies Assumption~\ref{ass:N1_alt} or~\ref{ass:N2_alt}, as appropriate.

    \item 
For $\scale=(1,0)$, $\ALPHABET M_{\boldsymbol{\alpha}} = \ALPHABET M$ and hence, the bounds in Theorems \ref{thm:alpha_beta} and~\ref{thm:alpha_beta_openloop} are identical to those in Theorems \ref{thm:bound} and \ref{thm:openloop bound}, respectively.
    \item Thus, if we optimize over appropriate $(\alpha_1, \alpha_2)$, then the bounds of Theorems~\ref{thm:alpha_beta} and~\ref{thm:alpha_beta_openloop} are tighter than those of Theorems~\ref{thm:bound} and~\ref{thm:openloop bound}. For instance, if $\hat {\ALPHABET M} = \ALPHABET M_{(2,1)}$, then the bounds of Theorems~\ref{thm:alpha_beta} and~\ref{thm:alpha_beta_openloop} are zero for $\boldsymbol{\alpha} = (2,1)$, while the bounds of Theorems~\ref{thm:bound} and~\ref{thm:openloop bound} may be positive.

\end{itemize}

\section{Some instances of the main results}

\subsection{Inventory management}\label{sec:inventory}
In this section, we illustrate the results of Theorem~\ref{thm:bound} for an inventory management problem with state space $\ALPHABET S = \{ -S_{\max}, -S_{\max} + 1, \dots S_{\max} \}$ and action space $\ALPHABET A = \{0, 1, \dots, S_{\max} \}$. Let $S_t \in \ALPHABET S$ denote the amount of stock at the beginning of day~$t$, $A_t \in \ALPHABET A$ denote the stock ordered at the beginning of day~$t$, and $W_t \in \mathds{Z}_{\ge 0}$ denote the demand during day~$t$. The  dynamics are given by
\[
    S_{t+1} = \bigl[ S_t + A_t - W_t \bigr]_{-S_{\max}}^{S_{\max}}
\]
where $[ \cdot ]_{-S_{\max}}^{S_{\max}}$ denotes a function which clips its value between $-S_{\max}$ and $S_{\max}$. The demand $W_t$ is assumed to be an i.i.d.\ Binomial$(n,q)$ process. The per-step cost is given by 
\[
    c(s,a) = pa +  c_h s \IND_{\{s \ge 0\}} - c_s s\IND_{\{s < 0\}}
\]
where $c_h$ is the per-unit holding cost, $c_s$ is the per-unit shortage cost, and $p$ is the per-unit procurement cost. We denote the above model by $\ALPHABET M = (S_{\max}, \gamma, n, q, c_h, c_s, p)$. 

We consider two models:
\vspace*{-0.35\baselineskip}
\begin{tabbing}
    \quad $\bullet$ \= True model \hskip 1.5em \= $\ALPHABET M = (500, 0.75, 10, 0.4, 4.0, 2, 5)$. \\
    \quad $\bullet$ \> Approx.\ model \> $\hat {\ALPHABET M} = (500, 0.75, 10, 0.5, 3.8, 2, 5)$.
\end{tabbing}
\vspace*{-0.35\baselineskip}
Since both models have finite state and action spaces, Assumption~\ref{assm:DP-solvability} is satisfied (with any choice of  weight function). We choose the weight function to have a similar shape as the per-step cost. In particular, we take $w(s) = 1 + (1.5 \cdot 10^{-2}) \bigl[\hat{c}_h s \IND_{\{s \ge 0 \}} - \hat{c}_s s \IND_{\{s < 0\}}\bigr]$, where $\hat{c}_h$ and $\hat{c}_s$ denote the per-unit holding and shortage costs of the approximate model, respectively. We verify that Assumptions~\ref{assm:stability} and~\ref{ass:N1} are satisfied with $\kappa = 1.07$.

The weighted-norm bound of Theorem~\ref{thm:bound}, part~2 implies that
\begin{equation}
   V^{\hat \pi^\star}(s) 
   - 
    \frac{1}{1- \gamma \kappa} 
   \bigl[ \MISMATCH^{\hat \pi^\star} \hat V^\star +   \MISMATCH^\star \hat V^\star \bigr] 
   w(s)
   \le
   V^\star(s) 
   \le
   V^{\hat \pi^\star}(s).
   \label{eq:IM-bound-w}
\end{equation}
We compare these bounds with the sup-norm bounds obtained by taking $w \equiv 1$.

\begin{figure}[!t]
\centering
\begin{subfigure}{0.49\linewidth}
\includegraphics[width=\linewidth]{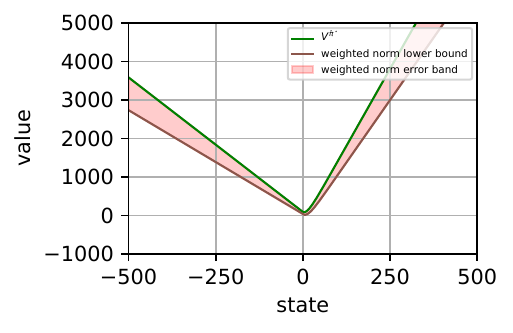}
\caption{weighted-norm bound}
\end{subfigure}\hfill
\begin{subfigure}{0.49\linewidth}
\includegraphics[width=\linewidth]{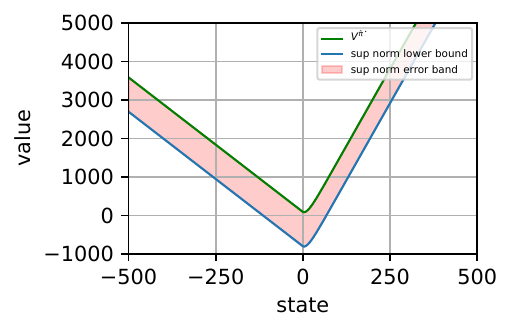}
\caption{sup-norm bound}
\end{subfigure}
\caption{Comparison of the bounds on $V^\star(s)$ based on weighted-norm and sup-norm.}
\label{fig:IM-bound}
\vskip \baselineskip
\begin{subfigure}{0.49\linewidth}
\includegraphics[width=\linewidth]{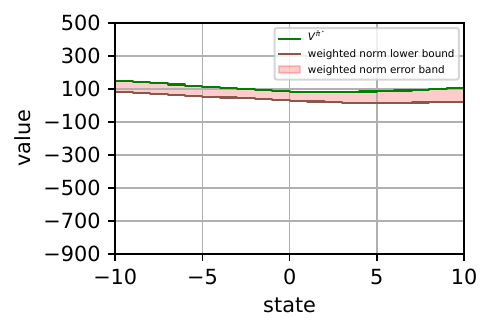}
\caption{weighted-norm bound}
\end{subfigure}\hfill
\begin{subfigure}{0.49\linewidth}
\includegraphics[width=\linewidth]{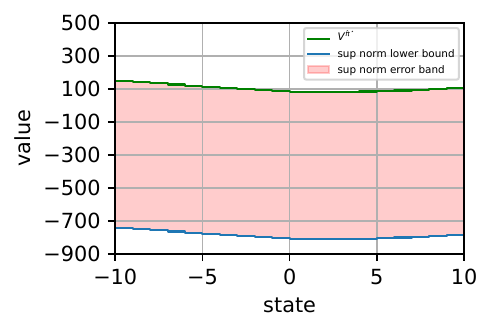}
\caption{sup-norm bound}
\end{subfigure}
\caption{Zoomed-in versions of the bounds of Fig.~\ref{fig:IM-bound}}
\label{fig:IM-bound-zoomed}
\end{figure}

For the models $\ALPHABET M$ and $\hat {\ALPHABET M}$ described above, we compute the policy $\hat \pi^\star$ using value iteration, compute $V^{\hat \pi^\star}$ using policy evaluation, and then plot the upper bound $V^{\hat \pi^\star}$, and  the weighted- and sup-norm lower bounds on $V^\star(s)$ given by the left hand side of~\eqref{eq:IM-bound-w} in Fig.~\ref{fig:IM-bound}. 

Fig.~\ref{fig:IM-bound} shows that the weighted-norm bound is slightly tighter than the sup-norm bound for most states. To better compare the error bounds, we zoom into the region of $\bar {\ALPHABET S} \coloneqq \{-10, -9, \dots, 10\}$ in Fig.~\ref{fig:IM-bound-zoomed}, where the weighted-norm bound is significantly better than the sup-norm bound.

The optimal policy for the inventory management model described above is a base-stock policy~\cite{arrow1951optimal}:
\(
   \pi^\star(s) = \max(0,\allowbreak {s^\star - s}),
\)
 where there is an optimal base-stock level $s^\star$ and whenever the inventory is less than $s^\star$, the optimal action is to order goods so that the inventory becomes $s^\star$. 
For the model $\hat {\ALPHABET M}$, the base-stock level $s^\star = 2$. Since the demand has finite support of $\{0, 1, \dots, 10\}$, after an initial transient period, the inventory level always remains between $\{-8, -7, \dots, 2 \}$. Thus, we care about the performance of an approximate policy in this region and, here, the weighted-norm bounds are substantially tighter than the sup-norm bounds.  These results show that even for finite state and action spaces, weighted-norm bounds can be better than sup-norm bounds. 

\subsection{Initial State dependent weight function}\label{sec:weight}
 Suppose there is a family $\ALPHABET W$ of weight functions such that for every $w \in \ALPHABET W$, there exists a $\kappa_w < 1/\gamma$ such that $(\kappa_w, w)$ satisfies Assumption~\ref{assm:stability}. Then, we can strengthen the result of~\eqref{eq:V-bound} as follows:
   \begin{equation}\label{eq:Vw-bound}
    V^{\hat \pi^\star}(s) - \inf_{w \in \ALPHABET W} \Bigl\{ \bigl\| V^{\hat \pi^\star} - V^\star \bigr\|_w w(s) \Bigr\}
    \le
    V^\star(s)
    \le
    V^{\hat \pi^\star}(s)
    .
    \end{equation}
    Note that the choice of weight function that gives the tightest bound  can vary with the start state~$s$. We illustrate the benefit of such a state dependent choice of weight function for the inventory management model of the previous section.

We consider the following family of weight functions for the inventory management model: 
\begin{equation}\label{eq:W}
   \ALPHABET W = \big\{1 + \ell \bar c(s) : \ell \in \{0, 0.5 \cdot 10^{-2}, 10^{-2}, \dots, 2.5 \cdot 10^{-2} \} \bigr\},
\end{equation}
where $\bar c(s) = \hat{c}_h s \IND_{\{s \ge 0 \}} - \hat{c}_s s \IND_{\{s < 0\}}$. 
Note that for $\ell = 0$, $w(s) = 1$ which corresponds to the sup-norm. 
For each $w \in \ALPHABET W$, we compute the smallest $\kappa_w$ such that Assumption~\ref{assm:stability} is satisfied as per~\eqref{eq:kappa-bound} and further verify that this value satisfies Assumption~\ref{ass:N1}. We plot the  lower bounds on $V^\star(s)$ corresponding to each $w \in \ALPHABET W$   in Fig.~\ref{fig:w-bound}. As can be seen from the figure, the best choice of weight function depends on the state. Minimizing over all $w \in \ALPHABET W$ as per~\eqref{eq:Vw-bound} gives a tighter bound. This tighter lower bound is highlighted in Fig.~\ref{fig:w-bound} using the shaded area shown in red.

\begin{figure}[!t]
\centering
\begin{subfigure}{0.49\linewidth}
\includegraphics[width=\linewidth]{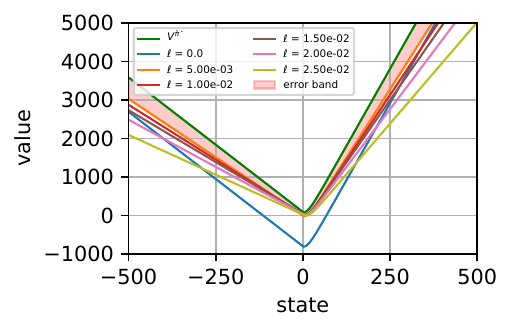}
\caption{entire state space}
\end{subfigure}\hfill
\begin{subfigure}{0.46\linewidth}
\includegraphics[width=\linewidth]{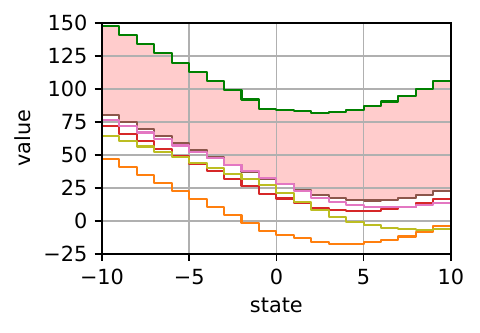}
\caption{zoomed in}
\end{subfigure}
\caption{Lower bounds obtained by different weight functions. Note that the curve corresponding to $\ell = 0$ is not visible in the zoomed in plot (b).}
\label{fig:w-bound}
\end{figure}

\subsection{Generalized bounds based on cost transformation}

The generalized bounds of Theorem~\ref{thm:alpha_beta}, part~2, imply that
\begin{multline}
   V^{\hat \pi^\star}(s) - 
    \frac{1}{\alpha_1(1- \gamma \kappa)} 
   \bigl[ \MISMATCH^{\hat \pi^\star}_{\boldsymbol{\alpha}} \hat V^\star +   \MISMATCH^\star_{\boldsymbol{\alpha}} \hat V^\star \bigr] 
   w(s)
   \le
   V^\star(s)
   \\
   \le
   V^{\hat \pi^\star}(s)
   \label{eq:alpha_beta_IM-bound-w}
\end{multline}
To show that this bound can be better than that of~\eqref{eq:IM-bound-w} obtained from Theorem~\ref{thm:bound}, part~2,
we consider the setup of Sec.~\ref{sec:weight} with $\ell = 1.5 \times 10^{-2}$ and compare $\boldsymbol{\alpha} = (0.98,0.8)$ with $\boldsymbol{\alpha} = (1,0)$. We verify that the appropriate assumptions are satisfied and plot the two bounds in Fig.~\ref{fig:alpha_beta-bound}.
As can be seen from the plots, the bound corresponding to Theorem~\ref{thm:alpha_beta} is tighter. 
\begin{figure}[!h]
\centering
\begin{subfigure}{0.49\linewidth}
\includegraphics[width=\linewidth]{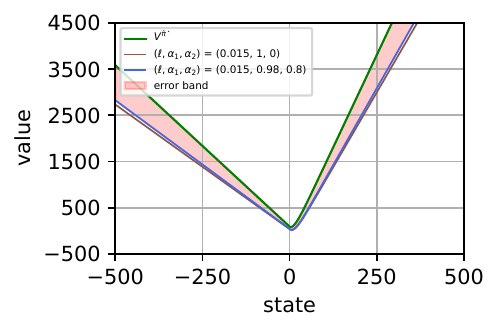}
\caption{entire state space}
\end{subfigure}\hfill
\begin{subfigure}{0.46\linewidth}
\includegraphics[width=\linewidth]{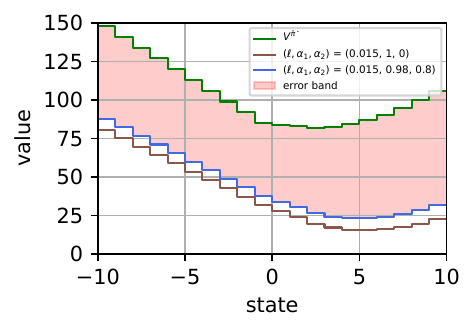}
\caption{zoomed in}
\end{subfigure}
\caption{Lower bounds obtained by different choices of $\boldsymbol{\alpha}$.}
\label{fig:alpha_beta-bound}
\end{figure}

\subsection{Linear quadratic regulator}\label{sec:LQR}

In this section, we use the linear quadratic regulator (LQR) to show that weighted norm bounds of Theorem~\ref{thm:bound} provide meaningful results for models with unbounded per-step cost. 
Consider a LQR problem with state space $\ALPHABET S = \reals^{n_s}$ and action space $\ALPHABET A = \reals^{n_a}$. The dynamics are given by 
\[
    s_{t+1}  = A s_t + B a_t + w_t,
\]
where $A$ and $B$ are system matrices of appropriate dimensions and $\{w_t\}_{t \ge 1}$ is an i.i.d.\ zero-mean noise process with covariance $\Sigma_W$. The per-step cost is given by 
\[
    c(s_t, a_t)  = s_t^\intercal Q s_t + a_t^\intercal R a_t,
\]
where $Q$ and $R$ are, respectively, positive semidefinite and positive definite matrices of appropriate dimensions. We will denote this model by $\ALPHABET M = (A,B,Q,R, \Sigma_W, \gamma)$ where $\gamma$ is the discount factor. 

Under standard assumptions of stabilizability and detectability, it is known that the optimal value function is
\[
    V^\star(s) = s^\intercal P s + q,
\]
where $P$ is the unique positive semidefinite solution of the discounted Riccati equation 
\begin{equation}
    P =  Q + \gamma A^\intercal P A - \gamma^2 A^\intercal PB(R + \gamma B^\intercal P B)^{-1} B^\intercal P A,
\end{equation}
and  
\(
    q = \gamma \Tr(\Sigma_W P)/(1-\gamma). 
\) Furthermore, the optimal policy is given as $\pi^\star(s) = -K^\star s$ where $K^\star = \gamma ( R + \gamma  B^\intercal  P B)^{-1}B^\intercal  P A$ is the optimal gain matrix \cite{bertsekas2015dynamic}.

We consider two models, a true model $\ALPHABET M = (A, B, Q, R, \Sigma_W, \gamma)$ and an approximate model $\hat{\ALPHABET M}= (\hat A, \hat B, \hat Q, \hat R, \hat \Sigma_W, \gamma)$. We take the weight function to be $w(s) = 1 + \ell s^\intercal s$, where $\ell > 0$ is a parameter. Under standard conditions of stabilizability and detectability (see~\cite{bertsekas2015dynamic}), both models $\ALPHABET M$ and $\hat {\ALPHABET M}$ satisfy Assumption~\ref{assm:DP-solvability}. Let $P$ and $\hat P$ denote the solutions of the Riccati equations corresponding to models $\ALPHABET M$ and $\hat {\ALPHABET M}$, and let $\pi^\star(s) = -K^\star s$ and $\hat \pi^\star(s) = -\hat K^\star s$ denote the optimal policies of models ${\ALPHABET M}$ and  $\hat {\ALPHABET M}$.

 For any linear policy $\pi(s) = -Ks$, we use the notation $A_K = A - BK$ and $\hat A_K = \hat A - \hat B K$. We further use $K_{\mu^\star}$ to denote the gain matrix of the (unique)  policy  ${\mu^\star} \in \ALPHABET{G}(\hat V^\star)$. We impose the following assumption.
\begin{assumption}\label{ass:LQ}
    The models $\ALPHABET M$ and $\hat {\ALPHABET M}$ are such that
    \[
        b_\Sigma \coloneqq \max\bigl\{
        1 + \ell \Tr(\Sigma_W),
        1 + \ell \Tr(\hat \Sigma_W)
        \bigr\} \le \frac 1\gamma
    \]
    and
    \[
       b_\sigma \coloneqq \max\bigl\{
            \sigma_1^2(A_{K^\star}),
            \sigma_1^2(A_{\hat K^\star}),
            \sigma_1^2(\hat A_{\hat K^\star}),
            \sigma_1^2(A_{K_{\mu^\star}})
       \bigr\} \le \frac 1\gamma 
    \]
    where $\sigma_1(A)$ is the operator norm of $A$ (i.e., the largest singular value of $A$). 
\end{assumption}

\begin{lemma}\label{lem:LQ}
    Assumption~\ref{ass:LQ} implies Assumptions~\ref{assm:stability} and~\ref{ass:N1}.
\end{lemma}
\begin{proof}
   Fix a policy $\pi(s) = Ks$. Eq.~\eqref{eq:c-pi} is always satisfied because
   \[
      \| c_{\pi} \|_w = \sup_{s\in \ALPHABET S} \frac{s^\intercal (Q + K^\intercal R K) s}{ 1 + \ell s^\intercal s}
      \le \frac{1}{\ell} \rho(Q + K^\intercal R K) < \infty 
   \]
   where $\rho(\cdot)$ denotes the spectral radius of a matrix.
  \endgraf 
  Moreover largest value of $\kappa$ for which~\eqref{eq:kappa} is satisfied is given by~\eqref{eq:kappa-bound}, which simplifies to
    \begin{align*}
        \kappa_w &= \sup_{s\in \ALPHABET S}\frac{\EXP[w(s_{t+1}) | s_t = s]}{w(s)} \notag \\
        &= \sup_{s\in \ALPHABET S}\frac{1 + \ell \Tr(\Sigma_W) + \ell  s^\intercal A_K^\intercal A_K s}{
         1 + \ell  s^\intercal s
         } 
         \notag\\
         &\leq \max(1 + \ell\Tr(\Sigma_W), \sigma_1^2(A_K)).
    \end{align*}
    Thus, if Assumption~\ref{ass:LQ} holds, then Assumptions~\ref{assm:stability} and~\ref{ass:N1} hold with $\kappa \coloneqq \max\{b_\Sigma, b_\sigma\}$.
\end{proof}

Then, the result of Theorem~\ref{thm:alpha_beta}, part~2 simplifies as follows:
\begin{proposition} \label{prop:LQR}
   Under Assumptions~\ref{assm:DP-solvability} and~\ref{ass:LQ}, we have for $\alpha_1 = 1$ and any $\alpha_2$,
   \begin{align}
        \bigl\lVert V^{\hat \pi^\star} - V^\star \bigr\rVert_w 
       \le
       & 
       \frac{1}{1  - \gamma\kappa} \bigl[
         \max\{ \rho(D^\star) / \ell, |d_\Sigma + \alpha_2| \}
         \notag\\
         & \quad 
         + \max \{ \rho(D^{\hat \pi^\star}) / \ell, |d_\Sigma + \alpha_2| \}
    \bigr],
    \label{eq:LQ-bound}
   \end{align}
   where $\rho(\cdot)$ denotes the spectral radius of a matrix and 
    \begin{align}
        D^\star &= 
        \bigl(
    Q + \gamma A^\intercal \hat P A - \gamma^2 A^\intercal \hat PB(R + \gamma B^\intercal \hat P B)^{-1} B^\intercal \hat P A
    \bigr)
        \notag \\
        & \quad - \hat P
        \label{eq:Dstar}
        \displaybreak[1]
    \\
    D^{\hat \pi^\star} &= 
    \big(
    Q +  (\hat K^\star)^\intercal R \hat K^\star + \gamma A_{\hat K^\star}^\intercal \hat P A_{\hat K^\star}
    \big) 
    - \hat P
    \label{eq:Dpihatstar}
    \displaybreak[1]
    \\
    \hat K^\star &= \gamma (\hat R + \gamma \hat B^\intercal \hat P \hat B)^{-1}\hat B^\intercal \hat P \hat A,
    \label{eq:Khat}
    \displaybreak[1]
    \\
        \shortintertext{and}
    d_\Sigma & = \gamma \Tr((\Sigma_W - \hat \Sigma_W) \hat P).
    \label{eq:dsigma}
    \end{align}
By taking $\alpha_2 = -d_{\Sigma}$ we obtain
    \begin{equation}
        \bigl\lVert V^{\hat \pi^\star} - V^\star \bigr\rVert_w 
    \le 
       \frac{1}{\ell(1  - \gamma\kappa)} \bigl[
         \rho(D^\star) + \rho(D^{\hat \pi^\star})
       \bigr].
       \label{eq:LQ-bound-simple}
   \end{equation}
\end{proposition}
See Appendix~\ref{append:lqr} for proof.
\begin{remark}
    Under Assumptions~\ref{assm:DP-solvability} and~\ref{ass:LQ}, the bound obtained in~\eqref{eq:LQ-bound-simple} does not depend on the Riccati solution $(P,K)$ of the true model $\ALPHABET M$.
\end{remark}

\begin{remark}
    Consider the case when $\hat {\ALPHABET M}$ is the same model as $\ALPHABET M$ except $\hat \Sigma_W = 0$. In this case, $D^\star = 0$ and $D^{\hat \pi^\star} = 0$. Therefore, from~\eqref{eq:LQ-bound-simple}, we get that $\| V^{\hat \pi^\star} - V^\star\|_{w} = 0$, which corresponds to the classical certainty equivalence principle of LQR control.
\end{remark}

\subsection{Advantage of using $(\tuple)$ stability of policy over $(\tuple)$ stability of model}\label{sec:pi_better}
As mentioned in Remark~\ref{rem:pi-better}, it is typically assumed in the literature that the model is $(\bar \kappa, \bar w)$ stable, while we impose a weaker assumption that certain policies are $(\kappa, w)$ stable. In this section, we illustrate two advantages of imposing the weaker assumption.

First, when the per-step cost is unbounded in the actions, as is the case for the LQR problem considered in Sec.~\ref{sec:LQR}, the model $\ALPHABET M$ is not $(\bar \kappa, \bar w)$ stable for any choice of weight function $\bar w$. However, as illustrated in Sec.~\ref{sec:LQR}, specific policies may be $(\kappa, w)$ stable for $w(s) = 1 + s^\intercal s$. Thus, imposing a weaker assumption of stability allows us to derive approximation bounds for a larger class of models.

Second, imposing a weak assumption of stability allows us to derive tighter approximation bounds. To illustrate this, we reconsider the 
inventory management problem in the setting of Sec.~\ref{sec:weight}. In Sec.~\ref{sec:weight}, we had computed the lower bounds of~\eqref{eq:IM-bound-w} when using $(\kappa,w)$ that satisfy Definition~\ref{def:stability}. Now, we consider $(\bar \kappa, \bar w)$ that satisfy Definition~\ref{def:model-stability} instead. In particular, consider 
a family of weight functions $\ALPHABET W$ as defined in~\eqref{eq:W}.
For each $\bar w\in \ALPHABET W$, we compute the smallest $\bar\kappa_{\bar w}$ such that \eqref{eq:stab_act} is satisfied as per~\eqref{eq:bar-kappa-bound}. The largest value of $\ell$ for which $\bar\kappa_{\bar w} < 1/{\gamma}$ is $\ell = 1.75 \cdot 10^{-4}$. 

We plot the corresponding lower bound given in~\eqref{eq:IM-bound-w} in Fig.~\ref{fig:bar-k-bound}. As can be seen from the plot, in this case the weight function $\bar w(s) \equiv 1$ (equivalent to the sup-norm) gives the tightest lower bound. But, as was seen by the bounds of Fig.~\ref{fig:w-bound}, the bounds obtained by weighted functions in class $\ALPHABET W$ were significantly tighter. This highlights the importance of imposing the weaker assumption of $(\tuple)$-stability of policy rather than the $(\bar \kappa, \bar w)$-stability of the model.

\begin{figure}[!t]
\centering
\begin{subfigure}{0.49\linewidth}
\includegraphics[width=\linewidth]{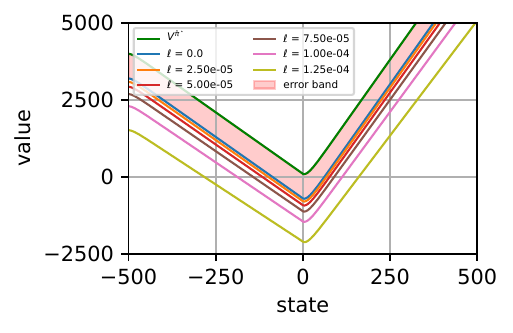}
\caption{entire state space}
\end{subfigure}\hfill
\begin{subfigure}{0.49\linewidth}
\includegraphics[width=\linewidth]{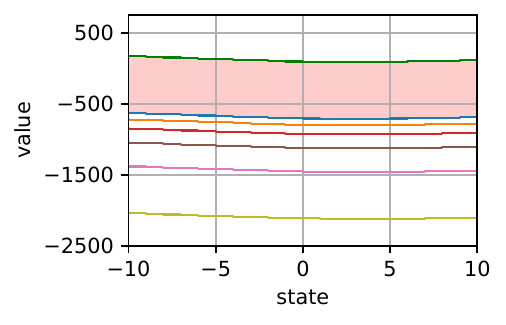}
\caption{zoomed in}
\end{subfigure}
\caption{Lower bounds obtained using stability of the model. Note that the curves corresponding to $\ell = 1.50 \cdot 10^{-4}$ and $\ell = 1.75\cdot 10^{-4}$ are not visible in both plots.}
\label{fig:bar-k-bound}
\end{figure}

\section{Integral probability metrics (IPM) and bounds based on distance between models} \label{sec:IPM}

In this section, we provide upper bounds for the results of Sec.~\ref{sec:main_results} that can be computed in terms of the \textit{distance} between models $\ALPHABET M$ and $\hat{\ALPHABET M}$. To define such a distance, we first start with the definition of integral probability metrics (IPM)~\cite{muller1997}.
\subsection{Integral probability metrics (IPM)}

\begin{definition}\label{def:IPM}
    Let $\F$ be a convex and balanced subset of $\ALPHABET V_w$. Then, 
    the IPM distance (w.r.t.\ $\F$) between two probability laws $\nu_1, \nu_2 \in \ALPHABET P_w$ is given by\footnote{Since $\nu_i \in \ALPHABET P_w$, $i \in \{1,2\}$, we have $\int f d\nu_i < \infty$
    for any $f \in \ALPHABET V_w$.}
    \[
        d_{\F}(\nu_1, \nu_2) = \sup_{f \in \mathfrak {F}}\biggr|\int f d\nu_1 - \int f d\nu_2 \biggl|.
    \]
\end{definition}
\begin{definition}
    In the setting of Definition~\ref{def:IPM}, the  Minkowski functional of any measurable function $f \in \ALPHABET V_w$ is defined as
    \[
        \rho_{\F}(f) = \inf\Bigl\{\rho \in \mathbb R_{>0}: \frac{f}{\rho}\in \F\Bigr\}.
    \]
\end{definition}
Note that if for every positive $\rho$, $f/\rho \not\in \F$, then $\rho_{\F}(f) = \infty$.

An immediate consequence of the above two definitions is that for any measurable function $f \in \ALPHABET V_w$,
\begin{equation}\label{eq:IPM_main}
    \biggl| \int f d\nu_1 - \int f d\nu_2 \biggr| \leq \rho_{\F}(f)d_{\F}(\nu_1,\nu_2).
\end{equation}
Many of the commonly used metrics on probability spaces are IPMs. For example

\begin{itemize}
    \item \textbf{Total variation distance}, denoted by $\dTV$, corresponds to $\F = \FTV \coloneqq \{ f \in \ALPHABET V_{w\equiv1} : \tfrac 12\Span(f) \le 1 \}$, where $\Span(f) = \sup(f) - \inf(f)$~\cite{muller1997,villani2009optimal}. In this case $\rho_{\F}(f) = \tfrac 12 \Span(f)$.
    \item \textbf{Wasserstein distance}. Suppose $(\ALPHABET S, d_{\ALPHABET S})$ is a metric space  
    Define $\FWasw \coloneqq \{ f \in \ALPHABET V_w : \Lip(f) \le 1 \}$ where  $\Lip(f)$ denotes the Lipschitz constant of a function $f$. 
    Then, for $\F = \FWasw$, Eq~\eqref{eq:IPM_main} holds for $\rho_{\F}(f) = \Lip(f)$. Moreover $d_{\F}(\nu_1, \nu_2) \le \dWas(\nu_1, \nu_2)$, where $\dWas$ is the Wasserstein distance~\cite{vaserstein1969markov,villani2009optimal}. 
    \item \textbf{Weighted total variation distance}, denoted by $\dTVw$, corresponds to $\F = \FTVw \coloneqq \{ f \in \ALPHABET V_w: \Osc_w(f) \le 1 \}$, where $\Osc_w(f) = \sup_{s,s' \in \ALPHABET S \times \ALPHABET S}| f(s) - f(s')|/(w(s) + w(s'))$~\cite{douc2018markov,hairer2011yet,meyn1992stability}. For this case, $\rho_{\F}(f) = \Osc_w(f)$. 
\end{itemize}

\subsection{Weighted distance between two MDP models}
Note that if a policy $\pi$ is $(\tuple)$ stable, then Eq.~\eqref{eq:kappa} implies that $P_\pi(\cdot | s) \in \ALPHABET P_w$ for every $s \in \ALPHABET S$. Therefore, Assumption~\ref{assm:stability} implies that
for all $s \in \ALPHABET S$,  $P_{\pi^\star}(\cdot | s), P_{\hat \pi^\star}(\cdot | s), \hat P_{\hat \pi^\star}(\cdot | s) \in \ALPHABET P_w$ and Assumption~\ref{assm:openloop} implies that for all $(s,a) \in \ALPHABET S \times \ALPHABET A$, we have  $P(\cdot | s,a), \hat P(\cdot | s,a) \in \ALPHABET P_w$.

We now define two notions of weighted distance between two MDP models.
\begin{definition}[Distance between MDP models]\label{def:model_distance}
    Given two MDP models $\ALPHABET M = \langle \ALPHABET S, \ALPHABET A, P, c, \gamma \rangle$ and $\hat{\ALPHABET M} = \langle \ALPHABET S, \ALPHABET A, \hat P, \hat c, \gamma \rangle$, a weight function $w: \ALPHABET S \to [1,\infty)$ and an IPM $d_{\F}$ as defined in Definition~\ref{def:IPM}, we define the following
    
    \begin{enumerate}
        \item \textbf{Distance between models for given policies:} Given deterministic policies $\pi$ for model $\ALPHABET M$ and $\hat\pi$ for model $\hat{\ALPHABET M}$ such that $P_\pi(\cdot | s), \hat P_{\hat \pi}(\cdot | s) \in \ALPHABET P_w$ for all $s \in \ALPHABET S$, define
        \begin{align*}
            \varepsilon_{\boldsymbol{\alpha}}(\pi,\hat\pi) 
            &\coloneqq \sup_{s \in \ALPHABET S}\frac{\bigl|\alpha_1c_{\pi}(s)+\alpha_2-\hat c_{\hat\pi}(s) \bigr|}{w(s)}
        \end{align*}
        and 
        \[
            \delta_{\mathfrak F}(\pi,\hat\pi) \coloneqq
            \sup_{s \in \ALPHABET S}\frac{d_{\mathfrak F}\bigl(P_{\pi}(\cdot|s\bigr),\hat P_{\hat\pi}\bigl(\cdot|s)\bigr)}{w(s)}.
        \]
        
        \item \textbf{Maximal distance between models}: Under Assumption~\ref{assm:openloop}, define
        \[
            \varepsilon^{\max}_{\boldsymbol{\alpha}} \coloneqq \sup_{(s,a)\in \ALPHABET S \times \ALPHABET A} \frac{\left|\alpha_1c(s,a)+\alpha_2-\hat c(s,a) \right|}{w(s)}
        \]
        and
        \[
            \delta_{\mathfrak F}^{\max} \coloneqq \sup_{(s,a)\in \ALPHABET S \times \ALPHABET A} 
            \frac{d_{\mathfrak F}\bigl(P(\cdot|s,a),\hat P(\cdot|s,a)\bigr)}{w(s)}.
        \]
    \end{enumerate}
\end{definition}

Note that the distances defined above depend on the weight function $w$, but we don't explicitly capture that dependence in the notation.

\subsection{IPM based approximation bounds}

\begin{lemma}\label{lem:mismatch_bounds}
 We have the following bounds for different mismatch functionals:
    \begin{enumerate}
    \setlength\itemsep{1em}
        \item If policies $\pi$ and $\hat\pi$ are such that for all $s \in \ALPHABET S$, $P_{\pi}(\cdot|s)$, $\hat P_{\pi}(\cdot|s)$, $P_{\hat \pi}(\cdot|s) \in \ALPHABET P_w$, then for all $v \in \ALPHABET V_w$,
        
        \[\MISMATCH^{\pi,\hat\pi}_{\boldsymbol{\alpha}}v \leq 
        \varepsilon_{\boldsymbol{\alpha}}(\pi,\hat\pi) + \gamma \rho_{\mathfrak F}(v)\delta_{\mathfrak F}(\pi,\hat\pi).\]
        
        \item If Assumption~\ref{ass:N1_alt} is satisfied, 
        \begin{equation*}
            \MISMATCH^{\star}_{\boldsymbol{\alpha}}\hat V^\star 
            \leq \varepsilon_{\boldsymbol{\alpha}}(\mu^\star,\hat \pi^\star) + \gamma \rho_{\mathfrak F}(\hat V^\star)\delta_{\mathfrak F}(\mu^\star,\hat \pi^\star),
        \end{equation*}
        for all $\mu^\star \in \ALPHABET G_{\boldsymbol{\alpha}}(\hat V^\star) \cap  \Pi_S(\tuple)$ (see Assumption~\ref{ass:N1_alt}).
        \item If Assumption~\ref{ass:N2_alt} is satisfied, 
        \begin{equation*}
            \MISMATCH^{\star}_{\boldsymbol{\alpha}} (\alpha_1 V^\star) 
            \leq \varepsilon_{\boldsymbol{\alpha}}(\pi^\star,\hat \mu^\star) + \alpha_1 \gamma \rho_{\mathfrak F}(V^\star)\delta_{\mathfrak F}(\pi^\star,\hat \mu^\star),
        \end{equation*}
        for all $\hat\mu^\star \in \hat {\ALPHABET G}(V^\star_{\boldsymbol{\alpha}}) \cap  \hat \Pi_S(\tuple)$ (see Assumption~\ref{ass:N2_alt}).
        
        \item If Assumption~\ref{assm:openloop} is satisfied, then for all $v \in \ALPHABET V_w$,
        \[
        \MISMATCH^{\max}_{\boldsymbol{\alpha}}v \leq 
        \varepsilon^{\max}_{\boldsymbol{\alpha}} + \gamma \rho_{\mathfrak F}(v)\delta_{\mathfrak F}^{\max}.
        \]
    \end{enumerate}
\end{lemma}
See Appendix~\ref{app:mismatch_bounds} for proof.

Substituting the results of Lemma~\ref{lem:mismatch_bounds} in Theorem~\ref{thm:alpha_beta} and Theorem~\ref{thm:alpha_beta_openloop}, we obtain the following:

\begin{theorem}\label{thm:weighted_IPM}
    We have the following bounds on $V^{\hat \pi^\star} - V^\star$:
    \begin{enumerate}
        \item Under Assumptions~\ref{assm:DP-solvability} and~\ref{assm:stability}, we have
        \begin{align*}
                \hskip -1em \bigl\| V^{\hat \pi^\star} - V^\star \bigr\|_w &\leq 
                \frac{1}{\alpha_1(1-\gamma\kappa)}\biggl[ \varepsilon_{\boldsymbol{\alpha}}(\hat\pi^\star,\hat\pi^\star) +\varepsilon_{\boldsymbol{\alpha}}(\pi^\star,\hat\pi^\star)   \\ 
                &\quad +\gamma \rho_{\F}(\hat V^\star) \bigl( \delta_{\mathfrak F}(\hat\pi^\star,\hat\pi^\star) 
                +\delta_{\mathfrak F}(\pi^\star,\hat\pi^\star)\bigr)  \biggr]
        \end{align*}
        and 
        \begin{align*}
            &
            \bigl\| V^{\hat \pi^\star} - V^\star \bigr\|_w \\
            & \leq
            \frac{1}{\alpha_1(1-\gamma\kappa)}\bigl[\varepsilon_{\boldsymbol{\alpha}}(\hat\pi^\star,\hat\pi^\star) 
             +\alpha_1\gamma\rho_{\mathfrak F}(V^\star)\delta_{\mathfrak F}(\hat\pi^\star,\hat\pi^\star)\bigr] \\
            & + \frac{1+\gamma\kappa}{\alpha_1(1-\gamma\kappa)^2}\bigl[\varepsilon_{\boldsymbol{\alpha}}(\pi^\star,\hat\pi^\star) 
             + \alpha_1\gamma\rho_{\mathfrak F}(V^\star)\delta_{\mathfrak F}(\pi^\star,\hat\pi^\star)\bigr].
        \end{align*}
        \item Under Assumptions~\ref{assm:DP-solvability},~\ref{assm:stability} and~\ref{ass:N1_alt}, we have
        \begin{align*}
                \hskip -1em \bigl\| V^{\hat \pi^\star} - V^\star \bigr\|_w &\leq 
                \frac{1}{\alpha_1(1-\gamma\kappa)}\biggl[ \varepsilon_{\boldsymbol{\alpha}}(\hat\pi^\star,\hat\pi^\star) +\varepsilon_{\boldsymbol{\alpha}}(\mu^\star,\hat\pi^\star)   \\ 
                &\quad +\gamma \rho_{\F}(\hat V^\star) \bigl( \delta_{\mathfrak F}(\hat\pi^\star,\hat\pi^\star) 
                +\delta_{\mathfrak F}(\mu^\star,\hat\pi^\star)\bigr)  \biggr]
        \end{align*}
        for all $\mu^\star \in \ALPHABET G_{\boldsymbol{\alpha}}(\hat V^\star) \cap  \Pi_S(\tuple)$ (see Assumption~\ref{ass:N1_alt}).
        \item Under Assumptions~\ref{assm:DP-solvability},~\ref{assm:stability} and~\ref{ass:N2_alt}, we have
            \begin{align*}
            &
            \bigl\| V^{\hat \pi^\star} - V^\star \bigr\|_w \\
            & \leq
            \frac{1}{\alpha_1(1-\gamma\kappa)}\bigl[\varepsilon_{\boldsymbol{\alpha}}(\hat\pi^\star,\hat\pi^\star) 
             +\alpha_1\gamma\rho_{\mathfrak F}(V^\star)\delta_{\mathfrak F}(\hat\pi^\star,\hat\pi^\star)\bigr] \\
            & + \frac{1+\gamma\kappa}{\alpha_1(1-\gamma\kappa)^2}\bigl[\varepsilon_{\boldsymbol{\alpha}}(\pi^\star,\hat\mu^\star) 
             + \alpha_1\gamma\rho_{\mathfrak F}(V^\star)\delta_{\mathfrak F}(\pi^\star,\hat\mu^\star)\bigr].
        \end{align*}
        for all $\hat\mu^\star \in \hat {\ALPHABET G}(V^\star_{\boldsymbol{\alpha}}) \cap  \hat \Pi_S(\tuple)$ (see Assumption~\ref{ass:N2_alt}).
        \item Under Assumptions~\ref{assm:DP-solvability},~\ref{assm:stability} and~\ref{assm:openloop}, we have
        \[
            \bigl\| V^{\hat \pi^\star} - V^\star \bigr\|_w \leq \frac{2}{\alpha_1(1-\gamma\kappa)}
            \Bigl[ \varepsilon_{\boldsymbol{\alpha}}^{\max} + \gamma \rho_{\F}(\hat V^\star) \delta_{\F}^{\max} \Bigr]
        \]      
        and 
        \begin{align*}
            \hskip -1em \bigl\| V^{\hat \pi^\star} - V^\star \bigr\|_w \leq&
            \frac{2} {\alpha_1(1-\gamma\kappa)^2} \Bigl[ \varepsilon_{\boldsymbol{\alpha}}^{\max} + \alpha_1\gamma \rho_{\F}(V^\star) \delta_{\F}^{\max} \Bigr]. 
        \end{align*}
        
    \end{enumerate}
\end{theorem}

The bounds of Theorem~\ref{thm:weighted_IPM} may be specialized for specific choices of IPMs. We present these bounds in terms of $(\varepsilon_{\boldsymbol{\alpha}}^{\max}, \delta_{\F}^{\max})$ and $\hat V^\star$. The bounds in terms of 
$(\varepsilon_{\boldsymbol{\alpha}}(\pi^\star,\hat\pi^\star), \delta_{\mathfrak F}(\pi^\star,\hat\pi^\star))$ etc.\ and/or $V^{\star}$ can be expressed in a similar manner.

\begin{corollary}\label{cor:AIS_sup}
    Under Assumptions~\ref{assm:DP-solvability},~\ref{assm:stability} and~\ref{assm:openloop}, we have the following bounds on $V^{\hat \pi^\star} - V^\star$:
    \begin{enumerate}
        \item Bound in terms of total-variation distance:
        \begin{align*}
          \bigl\| V^{\hat \pi^\star} - V^\star \bigr\|_{w}
          &\le
          \frac{2}{\alpha_1(1 - \gamma\kappa) }
          \biggl[ \varepsilon_{\boldsymbol{\alpha}}^{\max} \\
          &  \hskip -2em
          + \gamma \sup_{\substack{ s \in \ALPHABET S \\ a \in \ALPHABET A}}
          \frac{\dTV (P (\cdot | s,a), \hat P(\cdot | s,a))}{w(s)} \, \frac{\Span(\hat V^\star)}{2}
          \biggr].
        \end{align*}
        \item Bound in terms of Wasserstein distance:
        \begin{align*}
          \bigl\| V^{\hat \pi^\star} - V^\star \bigr\|_w
          &\le
          \frac{2}{\alpha_1(1 - \gamma\kappa) }
          \biggl[ \varepsilon_{\boldsymbol{\alpha}}^{\max} \\
          &  \hskip -2em
          + \gamma \sup_{\substack{ s \in \ALPHABET S \\ a \in \ALPHABET A}}
          \frac{\dWas(P (\cdot | s,a), \hat P(\cdot | s,a))}{w(s)} \, \Lip(\hat V^\star)
          \biggr].
        \end{align*}
        \item Bound in terms of weighted total variation distance:
        \begin{align*}
          \bigl\| V^{\hat \pi^\star} - V^\star \bigr\|_{w}
          &\le
          \frac{2}{\alpha_1(1 - \gamma\kappa) }
          \biggl[ \varepsilon_{\boldsymbol{\alpha}}^{\max} \\
          &  \hskip -2em
          + \gamma \sup_{\substack{ s \in \ALPHABET S \\ a \in \ALPHABET A}}
          \frac{\dTVw (P (\cdot | s,a), \hat P(\cdot | s,a))}{w(s)} \, \Osc_w(\hat V^\star)
          \biggr].
        \end{align*}
    \end{enumerate}
\end{corollary}

Parts 1 and 2 of Corollary~\ref{cor:AIS_sup} may be viewed as weighted generalization of approximation results presented in~\cite{muller1997,abel2016near,gelada2019deepmdp} (some of those results assumed that the approximate model has a smaller state space than the original model).

\subsection{Performance loss in using certainty equivalent control}

Certainty equivalence refers to the following design methodology to determine a control policy for a stochastic control problem. Replace the random variables  in the stochastic control problem by their (conditional) expectations, solve the resulting deterministic control problem to determine a feedback control policy, and use the resulting \emph{certainty equivalent control policy} in the original stochastic system~\cite{simon1956dynamic,theil1957note}. It is well known that for systems with linear dynamics and quadratic cost (LQ problems), certainty equivalent control policies are optimal. But this is not the case in general. In this section, we use the results of Theorem~\ref{thm:weighted_IPM} to characterize the performance loss when using certainty equivalence for general dynamics with additive noise. 

Consider a system with state space $\reals^n$, action space $\reals^m$, and dynamics 
\begin{equation}\label{eq:stochastic}
    S_{t+1} = f(S_t, A_t) + N_t
\end{equation}
where $f$ is a measurable function and $\{N_t\}_{t \ge 1}$ is a zero-mean i.i.d.\ noise sequence with control law $\nu_N$. The per-step cost is given by $c(S_t, A_t)$. 

Now consider a deterministic model obtained by assuming that the  noise sequence in~\eqref{eq:stochastic} takes its expected value, i.e., the dynamics are 
\begin{equation}\label{eq:deterministic}
    S_{t+1} = f(S_t, A_t).
\end{equation}
The per-step cost is the same as before. 

Let $\ALPHABET M$ denote the stochastic model and $\hat {\ALPHABET M}$ denote the deterministic model. Then, the certainty equivalent design is to use the control policy $\hat \pi^\star$ in original stochastic model $\ALPHABET M$. Suppose Assumptions~\ref{assm:DP-solvability}, \ref{assm:stability}, and~\ref{assm:openloop} are satisfied for some $(\tuple)$. We use the Wasserstein distance based bounds in Corollary~\ref{cor:AIS_sup} to bound $\| V^{\hat \pi^\star} - V^\star\|_w$, where we take $\boldsymbol{\alpha} = (1,0)$ for simplicity. We assume that there is some norm $\| \cdot \|$ on $\reals^n$ and the Wasserstein distance and Lipschitz constant are computed with respect to this norm.

Since the costs are the same for both models, $\varepsilon_{\boldsymbol{\alpha}}^{\max} = 0$. We now characterize $\delta^{\max}$. 
For ease of notation, given random variables $X$ and $Y$ with probability laws $\nu_X$ and $\nu_Y$, we will use $\dWas(X,Y)$ to denote $\dWas(\nu_X, \nu_Y)$. 
Wasserstein distance is defined as~\cite{villani2009optimal}
\begin{equation}\label{eq:Kantorovich}
    \dWas(\nu_X, \nu_Y) = \inf_{(\tilde X, \tilde Y) \colon \tilde X \sim \nu_X, \tilde Y \sim \nu_Y }
    \EXP[ \| \tilde X - \tilde Y \| ].
\end{equation}
Now, for a fixed $(s,a)$, define $X = f(s,a) + N$, where $N \sim \nu_N$, and $Y = f(s,a)$. Then, the Wasserstein distance between $P(\cdot | s,a)$ and $\hat P(\cdot | s,a)$ is equal to $\dWas(X,Y)$ which, by~\eqref{eq:Kantorovich}, equals $\EXP[\| N \|]$ for all $(s,a)$. Thus, 
\[
    \delta^{\max}_{\FWas} = \sup_{(s,a) \in \ALPHABET S \times \ALPHABET A} \frac{ \EXP[ \| N \| ] }{ w(s) } \leq \EXP[ \| N \| ] 
\]

Thus, by Corollary~\ref{cor:AIS_sup}, part~2, we get
\begin{equation}\label{eq:CE-bound}
    \| V^{\hat \pi^\star} - V^\star \|_w \le \frac{2\gamma}{1- \gamma \kappa} \EXP[ \| N \| ] \Lip(\hat V^\star).
\end{equation}
This bound precisely quantifies the engineering intuition that certainty equivalent control laws are good when the noise is ``small''.

\begin{remark}
    The right hand side of~\eqref{eq:CE-bound} does not depend on the weight function (provided the weight function satisfies Assumption~\ref{assm:stability}). Suppose  the per-step cost is such that $c_{\min} \ge 0$ and $w = 1 + V^\star$ satisfies Assumption~\ref{assm:stability} for some $\kappa < 1/\gamma$. Then, Eq.~\eqref{eq:CE-bound} implies that 
    \[
        V^\star(s) \le V^{\hat \pi^\star}(s) \le 
        \biggl(1 + \frac{2\gamma}{1- \gamma \kappa} \EXP[ \| N \| ] \Lip(\hat V^\star) \biggr) \bigl(1 +V^\star(s)\bigr).
    \]
    This inequality may be viewed as a generalization of the approximation bounds of~\cite{Witsenhausen1969} to dynamical systems.
\end{remark}

\section{Conclusion}
  In this paper, we present a series of bounds on the weighted approximation error when using the optimal policy of an approximate model $\hat {\ALPHABET M}$ in the original model $\ALPHABET M$. For each bound, we have two types of bounds: one which depends on the value function $\hat V^\star$ of the approximate model $\hat {\ALPHABET M}$ and the other which depends on the value function $V^\star$ of the original model $\ALPHABET M$. The first type of bound is more useful in practice because one would obtain $\hat V^\star$ when computing the optimal policy of the approximate model $\hat {\ALPHABET M}$. However, the second type of bound is a theoretical upper bound that may be useful for obtaining bounds for reinforcement learning algorithms, e.g., in obtaining sample complexity bounds. 

  Our results rely on using an appropriate $(\tuple)$ such that certain policies are $(\tuple)$ stable. The choice of the weight function $w$ impacts the tightness of the bounds.  Understanding how to choose weight functions is an interesting research direction. 

  In this paper, we assumed that the approximate model $\hat {\ALPHABET M}$ was given. However, often the approximate model is a design choice. For example, when solving continuous state models, we may decide how to quantize the state space. The approximation bounds obtained in this paper may be useful in guiding the design of such approximate models. They may also be useful in generalizing the convergence guarantees and regret bounds of reinforcement learning algorithms to models with unbounded per-step cost.

\bibliographystyle{IEEEtran}
\bibliography{IEEEabrv,references}
\appendices

\section{Proof of Lemma~\ref{lem:bounded}}\label{app:bounded}

We prove each part separately.

\subsubsection*{Proof of part 1)} Fix a state $s \in \ALPHABET S$. For a policy $\pi \in \Pi_S(\kappa,w)$ and a value function $v \in \ALPHABET V_w$, we have
\begin{align*}
    \hskip 1em & \hskip -1em 
    \left|\frac{\BELLMAN ^{\pi}v(s)}{w(s)}\right|
    \\ &\stackrel{(a)}\le \left|\frac{c_{\pi}(s)}{w(s)}\right| + \gamma \left| \int_{\ALPHABET S}P_{\pi}(ds' \mid s) \frac{v(s')}{w(s')} \frac{w(s')}{w(s)}\right| \\
    &\stackrel{(b)}{\leq}  \| c_\pi \|_w + \gamma \| v \|_w \left| \int_{\ALPHABET S}P_{\pi}(ds' \mid s) \frac{w(s')}{w(s)}\right|\\
    &\stackrel{(c)}{\le}  \| c_\pi \|_w + \gamma \| v \|_w \kappa <\infty,
\end{align*}
where $(a)$ follows from the triangle inequality, $(b)$ follows from the definition of $\| \cdot \|_{w}$ and $(c)$ follows from the fact that $\pi$ is $(\tuple)$ stable.
\subsubsection*{Proof of part 2)} Fix a state $s \in \ALPHABET S$. We have
\begin{align*}
    \hskip 0.5em & \hskip -0.5em \left|\frac{[\BELLMAN^{\pi} v_1 - \BELLMAN^{\pi}v_2](s)}{w(s)}\right|  \\
    \displaybreak[1]
     &= \gamma \left| \int_{\ALPHABET S}P_{\pi}(ds' \mid s)\biggl[\frac{v_1(s')-v_2(s')}{w(s')}\biggr]\frac{w(s')}{w(s)}\right| \\
     &\stackrel{(a)}{\leq}  \gamma \| v_1 - v_2 \|_w \left|\int_{\ALPHABET S}P_{\pi}(ds' \mid s)\frac{w(s')}{w(s)} \right|\\
     &\stackrel{(b)}{\leq}  \gamma \kappa \| v_1 - v_2 \|_w
\end{align*}
where $(a)$ holds from the definition of $\| \cdot \|_{w}$ and $(b)$ holds because $\pi$ is $(\tuple)$ stable.

\subsubsection*{Proof of part 3)}
From parts 1 and 2 of Lemma~\ref{lem:bounded}, we know that $\BELLMAN^\pi:\ALPHABET V_w \to \ALPHABET V_w $ is a contraction. Since $\ALPHABET V_w$ is a complete metric space (under the $\|\cdot\|_w$ norm), it follows from Banach fixed point theorem that $\BELLMAN^\pi$ has a unique fixed point $F$ in $\ALPHABET V_w$. If $V^{\pi}_n$ denotes the $n-$step discounted cost for policy $\pi$, then it can be shown that $V^{\pi}_{n+1} = \BELLMAN^\pi V^{\pi}_n$ and that $V^{\pi}_n \in \ALPHABET V_w$ for all $n$. Thus, by Banach fixed point theorem, $V^{\pi}_n$ converges to the fixed point $F$ of $\BELLMAN^\pi$ in the $\|\cdot\|_w$ norm. Since convergence in $\|\cdot\|_w$ norm implies pointwise convergence, we have $F(s) = \lim_{n \to \infty} V^\pi_n(s)$ for all $s \in \ALPHABET S$. Furthermore, since per-step costs are bounded below, we have that for all $s \in \ALPHABET S$,
\begin{align*}
\lim_{n \to \infty} V^\pi_n(s)
&= 
\lim_{n \to \infty} \EXP^{\pi} \biggl[ \sum_{t=1}^n \gamma^{t-1} c(S_t, A_t) \biggr]
\notag \\
&=
\EXP^{\pi} \biggl[ \lim_{n \to \infty} \sum_{t=1}^n \gamma^{t-1} c(S_t, A_t) \biggr]
= 
V^\pi(s)
\end{align*}
where the second equality follows from the monotone convergence theorem when $c_{\min} \ge 0$ (the case of $c_{\min} < 0$ follows from a similar argument by shifting $\{V^\pi_n\}_{n \ge 0}$ to make it non-negative and monotone).

\section{Proof of Lemma~\ref{lem:policy-error}}\label{app:policy-error}

Consider
\begin{align}
   \hskip 2em & \hskip -2em 
    \bigl\|V^{\pi} - \hat V ^{\hat\pi} \bigr\|_w =
    \bigl\|\BELLMAN ^{\pi}V^{\pi} - \hat \BELLMAN ^{\hat\pi} \hat V ^{\hat\pi}\bigr\|_w \notag \\
    &\leq \bigl\| \BELLMAN^{\pi}V^{\pi} - \hat \BELLMAN^{\hat\pi} V^{\pi}\bigr\|_w 
     + \bigl\|\hat \BELLMAN^{\hat\pi} V^{\pi} - \hat \BELLMAN ^{\hat\pi} \hat V ^{\hat\pi}\bigr\|_w \notag \\
     &\leq 
     \MISMATCH^{\pi,\hat\pi} V^{\pi}
     +
     \gamma \kappa \bigl\|V^{\pi} - \hat V ^{\hat\pi} \bigr\|_w 
     \label{eq:lem-mismatch-pt-1-step-3}
\end{align}
where the first inequality follows from triangle inequality, and the last from the definition of Bellman mismatch functional and Lemma~\ref{lem:bounded} as $\hat\pi \in \hat{\Pi}_S(\tuple)$. Re-arranging the terms in \eqref{eq:lem-mismatch-pt-1-step-3}, we obtain
\begin{equation}
     \bigl\|V^{\pi} - \hat V ^{\hat\pi} \bigr\|_w \leq
     \frac{1}{1-\gamma\kappa} \MISMATCH^{\pi,\hat\pi} V^{\pi}.
     \label{eq:lem-mismatch-pt-1-step-4}
\end{equation}
Next consider
\begin{align}
   \hskip 2em & \hskip -2em 
    \bigl\|V^{\pi} - \hat V ^{\hat\pi} \bigr\|_w =
    \bigl\|\BELLMAN ^{\pi}V^{\pi} - \hat \BELLMAN ^{\hat\pi} \hat V ^{\hat\pi}\bigr\|_w \notag \\
    &\leq \bigl\| \BELLMAN^{\pi}V^{\pi} - \BELLMAN^{\pi} \hat V^{\hat\pi}\bigr\|_w 
     + \bigl\| \BELLMAN^{\pi} \hat V^{\hat\pi} - \hat \BELLMAN ^{\hat\pi} \hat V ^{\hat\pi}\bigr\|_w \notag \\
     &\leq \gamma \kappa \bigl\|V^{\pi} - \hat V ^{\hat\pi} \bigr\|_w + \MISMATCH^{\pi,\hat\pi}\hat V^{\hat\pi}
     \label{eq:lem-mismatch-pt-1-step-1}
\end{align}
where the first inequality follows from triangle inequality, and the last from Lemma~\ref{lem:bounded} as $\pi \in \Pi_S(\tuple)$ and from the definition of Bellman mismatch functional. Re-arranging the terms in \eqref{eq:lem-mismatch-pt-1-step-1}, we obtain
\begin{equation}
     \bigl\|V^{\pi} - \hat V ^{\hat\pi} \bigr\|_w \leq
     \frac{1}{1-\gamma\kappa} \MISMATCH^{\pi,\hat\pi}\hat V^{\hat\pi}.
     \label{eq:lem-mismatch-pt-1-step-2}
\end{equation}

Combining \eqref{eq:lem-mismatch-pt-1-step-4} and \eqref{eq:lem-mismatch-pt-1-step-2} establishes \eqref{eq:DeltaVpi}.

\section{Proof of Lemma~\ref{lem:value-error}}\label{app:value-error}
\subsection{Proof of part 1}
For the bound in terms of $\hat V^\star$, we have
\begin{align}
    \|V^\star - \hat V^\star \|_w &= \|\BELLMAN^{\pi^\star} V^\star - \hat\BELLMAN^{\hat\pi^\star} \hat V^\star \|_w \notag \\
    &\leq \|\BELLMAN^{\pi^\star} V^\star - \BELLMAN^{\pi^\star} \hat V^\star \|_w + \|\BELLMAN^{\pi^\star} \hat V^\star - \hat\BELLMAN^{\hat\pi^\star} \hat V^\star \|_w \notag \\
    &\leq \gamma\kappa \|V^\star - \hat V^\star \|_w + \MISMATCH^{\pi^\star,\hat\pi^\star}\hat V^\star
    \label{eq:lem3_new_pf_1}
\end{align}
where the last inequality holds from Lemma~\ref{lem:bounded} as $\pi^\star \in \Pi_S(\tuple)$ and from the definition of Bellman mismatch functional. Re-arranging the terms in~\eqref{eq:lem3_new_pf_1}, we obtain
\begin{equation}
    \| V^\star - \hat V^\star \|_w \le \frac 1{(1 - \gamma \kappa)} \MISMATCH^{\pi^\star, \hat \pi^\star} \hat V^\star.
\end{equation}
For the bound in terms of $V^\star$, we have
\begin{align}
     \|V^\star - \hat V^\star \|_w &= \| \BELLMAN^{\pi^\star}V^\star - \hat\BELLMAN^{\hat\pi^\star}\hat V^\star \|_w \notag \\
     &\leq \|\BELLMAN^{\pi^\star}V^\star - \hat\BELLMAN^{\hat\pi^\star}V^\star\|_w + \| \hat\BELLMAN^{\hat\pi^\star}V^\star- \hat\BELLMAN^{\hat\pi^\star}\hat V^\star \|_w \notag \\
     &\leq \MISMATCH^{\pi^\star,\hat\pi^\star}V^\star + \gamma\kappa \|V^\star - \hat V^\star \|_w
     \label{eq:lem3_new_pf_2}
\end{align}
where the last inequality holds from the definition of Bellman mismatch functional and Lemma~\ref{lem:bounded} as $\hat\pi^\star \in \hat\Pi_S(\tuple)$. Re-arranging the terms in~\eqref{eq:lem3_new_pf_2}, we obtain
\begin{equation}
    \| V^\star - \hat V^\star \|_w \le \frac 1{(1 - \gamma \kappa)} \MISMATCH^{\pi^\star, \hat \pi^\star} V^\star.
\end{equation}
\subsection{Proof of part 2}
If $\BELLMAN^\star$ were a $\|\cdot\|_w$-norm contraction, then we could have used the exact same proof argument as  in proof of part~1. However, we have not established that $\BELLMAN^\star$ is a $\|\cdot\|_w$-norm contraction under Assumptions~\ref{assm:DP-solvability}, \ref{assm:stability} and \ref{ass:N1}. So, we need a different proof argument.
We use the shorthand notation $[ v ]_{w}$ to denote $\sup_{s \in \ALPHABET S} v(s)/w(s)$ (note that there is no absolute value sign around $v(s)$). 

Let $\mu^\star \in \ALPHABET G(\hat V^\star) \cap  \Pi_S(\tuple)$. Now, consider
\begin{align}
    [ V^\star - \hat V^\star ]_{w} 
    &=
    [ \BELLMAN^{\star} V^{\star} - \hat {\BELLMAN}^{\star} \hat V^{\star}]_{w}
    \notag \\
    &\stackrel{(a)}\le
    [ \BELLMAN^{\star} V^{\star} - \BELLMAN^{\star} \hat V^{\star}]_{w}
    +
    [ \BELLMAN^{\star} \hat V^{\star} - \hat {\BELLMAN}^{\star} \hat V^{\star}]_{w}
    \notag \\
    &\stackrel{(b)}\le 
    [ \BELLMAN^{\mu^\star} V^{\star} - \BELLMAN^{\mu^\star} \hat V^{\star}]_{w}
    + 
    \MISMATCH^{\star} \hat V^\star
    \notag \\
    &\stackrel{(c)}\le \gamma\kappa \| V^{\star} - \hat V^{\star} \|_w 
    + 
    \MISMATCH^{\star} \hat V^\star \label{eq:alt-proof-1}
\end{align}
where $(a)$ follows from the definition supremum, $(b)$ follows   from  $\mu^\star \in \ALPHABET G(\hat V^\star)$  and the fact that $\BELLMAN^\star  V^\star \le \BELLMAN^{\mu^\star} V^\star$ , and $(c)$ follows from contraction of the Bellman operator $\BELLMAN^{\mu^\star}$ (since $\mu^\star \in \Pi_S(\tuple)$). 

Now we consider the inequality in the other direction.
\begin{align}
    [ \hat V^\star - V^\star ]_{w} 
    &=
    [  \hat {\BELLMAN}^{\star} \hat V^{\star} -  \BELLMAN^{\star} V^{\star}  ]_{w}
    \notag \\
    &\stackrel{(d)}\le
    [ \hat {\BELLMAN}^{\star} \hat V^{\star} - \BELLMAN^{\star} \hat V^{\star}  ]_{w}
    +
    [ \BELLMAN^{\star} \hat V^{\star} - \BELLMAN^{\star} V^{\star}  ]_{w}
    \notag \\
    &\stackrel{(e)}\le 
    \MISMATCH^{\star} \hat V^\star
    + 
    [ \BELLMAN^{\pi^\star} \hat V^{\star} - \BELLMAN^{\pi^\star} V^{\star}  ]_{w}
    \notag \\
    &\stackrel{(f)}\le
    \MISMATCH^{\star} \hat V^\star 
    + 
    \gamma\kappa \| \hat V^{\star} - V^{\star} \|_w 
    \label{eq:alt-proof-2}
\end{align}
where $(d)$ follows from the definition of supremum, $(e)$ follows from $\pi^\star \in \ALPHABET G(V^\star)$ and the fact that $ {\BELLMAN}^\star \hat V^\star \le \BELLMAN^{\pi^\star} \hat V^\star$, and $(f)$ follows from contraction of the Bellman operator $\BELLMAN^{\pi^\star}$ (since $\pi^\star \in \Pi_S(\tuple)$).

Combining~\eqref{eq:alt-proof-1} and~\eqref{eq:alt-proof-2} and rearranging terms, we get~\eqref{eq:N1}.

\subsection{Proof of part 3}
The proof argument is similar to that of part~2. 
Consider
\begin{align}
    [ V^\star - \hat V^\star ]_{w} 
    &=
    [ \BELLMAN^{\star} V^{\star} - \hat {\BELLMAN}^{\star} \hat V^{\star}]_{w}
    \notag \\
    \displaybreak[1]
    &\stackrel{(a)}\le
    [ \BELLMAN^{\star} V^{\star} - \hat {\BELLMAN}^{\star} V^{\star}]_{w}
    +
    [ \hat {\BELLMAN}^{\star} V^{\star} - \hat {\BELLMAN}^{\star} \hat V^{\star}]_{w}
    \notag \\
    \displaybreak[1]
    &\stackrel{(b)}\le 
    \MISMATCH^{\star} V^\star
    + 
    [ \hat \BELLMAN^{\hat \pi^\star} V^{\star} - \hat \BELLMAN^{\hat \pi^\star} \hat V^{\star}]_{w}
    \notag \\
    &\stackrel{(c)}\le 
    \MISMATCH^{\star} V^\star 
    + 
    \gamma\kappa \| V^{\star} - \hat V^{\star} \|_w 
    \label{eq:alt-proof-21}
\end{align}
where $(a)$ follows from the definition of supremum, $(b)$ follows the definition of $\hat \pi^\star$ and the fact that $\hat {\BELLMAN}^\star V^\star \le \hat {\BELLMAN}^{\hat \pi^\star} V^\star$, and $(c)$ follows from contraction of the Bellman operator $\hat \BELLMAN^{\hat \pi^\star}$. 

Let $\hat\mu^\star \in \hat {\ALPHABET G}(V^\star) \cap  \hat \Pi_S(\tuple)$. Now, we consider the inequality in the other direction.
\begin{align}
    [ \hat V^\star - V^\star ]_{w} 
    &=
    [  \hat {\BELLMAN}^{\star} \hat V^{\star} -  \BELLMAN^{\star} V^{\star}  ]_{w}
    \notag \\
    \displaybreak[1]
    &\stackrel{(d)}\le
    [ \hat {\BELLMAN}^{\star} \hat V^{\star} - \hat \BELLMAN^{\star} V^{\star}  ]_{w}
    +
    [ \hat \BELLMAN^{\star} V^{\star} - \BELLMAN^{\star} V^{\star}  ]_{w}
    \notag \\
    \displaybreak[1]
    &\stackrel{(e)}\le 
    [ \hat \BELLMAN^{\hat \mu^\star} \hat V^{\star} - \hat \BELLMAN^{\hat \mu^\star} V^{\star}  ]_{w}
    + 
    \MISMATCH^{\star} V^\star
    \notag \\
    &\stackrel{(f)}\le
    \gamma\kappa \| \hat V^{\star} -  V^{\star} \|_w 
    + 
    \MISMATCH^{\star} V^\star 
    \label{eq:alt-proof-22}
\end{align}
where $(d)$ follows from the definition of supremum, $(e)$ follows the definition of $\hat \mu^\star$ and the fact that $\hat {\BELLMAN}^\star \hat V^\star \le \hat \BELLMAN^{\hat \mu^\star} \hat V^\star$, and $(f)$ follows from contraction of the Bellman operator $\hat\BELLMAN^{\hat \mu^\star}$. 

Combining~\eqref{eq:alt-proof-21} and~\eqref{eq:alt-proof-22} and rearranging terms, we get~\eqref{eq:N2}.

\section{Proof of Theorem~\ref{thm:bound}}\label{app:thm:bound}
We prove each part separately.
\subsection{Proof of part~1}
For the bound in terms of $\hat V^\star$, by triangle inequality, we have
\begin{equation}\label{eq:thm-bound-pt-1-step-1}
    \bigl\| V^{\hat \pi ^\star} - V^\star \bigr\|_w \leq 
        \bigl\|V^{\hat \pi ^\star} - \hat V ^\star \bigr\|_w + \bigl\| \hat V ^\star - V ^\star \bigr\|_w
\end{equation}
Recall that $\hat V^\star = \hat V^{\hat \pi^\star}$. Since $\hat \pi^\star \in \Pi_S(\tuple) \cap \hat \Pi_S(\tuple)$, we can use Lemma~\ref{lem:policy-error} to bound the first term of~\eqref{eq:thm-bound-pt-1-step-1} by
\begin{align}\label{eq:thm-bound-pt-1-step-2}
    \bigl\|V^{\hat \pi ^\star} - \hat V ^\star \bigr\|_w 
    &\le 
    \frac{1}{1- \gamma\kappa} \MISMATCH^{\hat \pi^\star,\hat\pi^\star} \hat V^{\hat \pi^\star} =
    \frac{1}{1- \gamma\kappa} \MISMATCH^{\hat \pi^\star} \hat V^{\hat \pi^\star} \notag \\
    &=
    \frac{1}{1- \gamma\kappa} \MISMATCH^{\hat \pi^\star} \hat V^{\star} .
\end{align}
We can bound the second term in~\eqref{eq:thm-bound-pt-1-step-1} using Lemma~\ref{lem:value-error}, part~1 by
\begin{equation}\label{eq:thm-bound-new-step}
    \|\hat V^\star-V^\star  \|_w \le \frac 1{(1 - \gamma \kappa)} \MISMATCH^{\pi^\star, \hat \pi^\star} \hat V^\star.
\end{equation}
We obtain the result by combining~\eqref{eq:thm-bound-pt-1-step-2} and~\eqref{eq:thm-bound-new-step}.

For the bound in terms of $V^\star$, we can write 
\begin{align}
   \hskip 2em & \hskip -2em
    \bigl\| V^{\hat\pi^\star} - V^\star \bigr\|_w =  \bigl\| \BELLMAN^{\hat\pi^\star}V^{\hat\pi^\star} - \BELLMAN^{\star}V^\star \bigr\|_w \notag \\
    &\leq \bigl\| \BELLMAN^{\hat\pi^\star}V^{\hat\pi^\star} - \BELLMAN^{\hat\pi^\star} V^\star \bigr\|_w
    + \bigl\|\BELLMAN^{\hat\pi^\star}V^{\star} - \hat\BELLMAN^{\hat\pi^\star} {V^{\star}} \bigr\|_w \notag \\
    &\quad + \bigl\| \hat\BELLMAN^{\hat\pi^\star}V^{\star} - \hat\BELLMAN^{\hat\pi^\star}\hat V^{\star}\bigr\|_w
    + \bigl\|\hat\BELLMAN^{\hat\pi^\star}\hat V^{\star} - \BELLMAN^{\star}V^{\star}\bigr\|_w \notag \\ 
    &\leq  \gamma\kappa \bigl\| V^{\hat\pi^\star} - V^\star \bigr\|_w + \MISMATCH^{\hat\pi^\star}V^{\star} \notag \\ 
    &\quad + \gamma\kappa \bigl\|V^{\star} - \hat V^{\star}\bigr\|_w + \bigl\| \hat V^{\star} - V^{\star}\bigr\|_w
    \label{eq:thm-bound-pt-2-step-1}
\end{align}
where the first inequality holds from triangle inequality and the last from the definition of Bellman mismatch functional and from Lemma~\ref{lem:bounded} as $\hat \pi^\star \in \Pi_S(\tuple) \cap \hat \Pi_S(\tuple)$. Re-arranging the terms in \eqref{eq:thm-bound-pt-2-step-1}, we obtain
\begin{align}
    \bigl\| V^{\hat\pi^\star} - V^\star \bigr\|_w \leq 
    \frac{1}{1-\gamma\kappa}\bigl[ \MISMATCH^{\hat\pi^{\star}}V^{\star} + (1+\gamma\kappa) \bigl\| \hat V^{\star} - V^\star\bigr\|_w \bigr].
    \label{eq:thm-bound-pt-2-step-2}
\end{align}
We use Lemma~\ref{lem:value-error}, part 1, to bound the last term of~\eqref{eq:thm-bound-pt-2-step-2} by 
\begin{equation}
    \| \hat V^\star-V^\star  \|_w \le \frac 1{(1 - \gamma \kappa)} \MISMATCH^{\pi^\star, \hat \pi^\star} V^\star.
    \label{eq:thm-bound-new-step-2}
\end{equation}
We obtain the result by combining~\eqref{eq:thm-bound-pt-2-step-2} and~\eqref{eq:thm-bound-new-step-2}.

\subsection{Proof of part~2}
Since Assumption~\ref{ass:N1} holds, we can use Lemma~\ref{lem:value-error} part~2 to
bound the second term of~\eqref{eq:thm-bound-pt-1-step-1} by
\begin{equation}\label{eq:thm-bound-pt-1-step-3}
    \bigl\| \hat V ^\star - V ^\star \bigr\|_w
    \le 
    \frac{1}{1- \gamma\kappa} \MISMATCH^{\star} \hat V^{\star} .
\end{equation}
Combining~\eqref{eq:thm-bound-pt-1-step-2} and~\eqref{eq:thm-bound-pt-1-step-3} completes the proof.

\subsection{Proof of part~3}
Since Assumption~\ref{ass:N2} holds, we can use Lemma~\ref{lem:value-error} part~3 to bound the last term of~\eqref{eq:thm-bound-pt-2-step-2} by
\begin{equation}
    \bigl\| \hat V^{\star} - V^\star\bigr\|_w \leq \frac{1}{1-\gamma\kappa}\MISMATCH^{\star}V^{\star}.
    \label{eq:thm-bound-pt-2-step-3}
\end{equation}
Combining \eqref{eq:thm-bound-pt-2-step-2} and \eqref{eq:thm-bound-pt-2-step-3} completes the proof.

\section{Proof of Lemma~\ref{lem:assm_5_4}}\label{app:assm_5_4}

Assumption~\ref{assm:openloop} implies that, for each $a \in \ALPHABET A$,  
\begin{equation}
  \int_{\ALPHABET S} w(s') P(ds' \mid s, a) 
  \le 
  \kappa w(s),
  \quad \forall s \in \ALPHABET S.
\end{equation}

For any $v \in \ALPHABET V_w$ such that $\ALPHABET G(v)$ is nonempty, let $\pi_v$ denote a policy in $\ALPHABET G(v)$. We will first show that $\pi_v \in \Pi_S(\tuple)$ and then use this to prove Assumption \ref{ass:N1}.

For policy $\pi_v$, we have that  
\begin{align*}
  \int_{\ALPHABET S} w(s') P_{\pi_v}(ds' | s)
   = &
   \int_{\ALPHABET S} 
   \int_{\ALPHABET A} w(s') \pi_v(da | s) P(ds' | s,a)
   \notag\\
   = &
   \int_{\ALPHABET A}
   \int_{\ALPHABET S} 
    w(s')  P(ds' | s,a)
    \pi_v(da | s)
   \notag\\
  \le &
  \int_{\ALPHABET A} \kappa w(s) \pi_v(da | s) = \kappa w(s),
  \quad \forall s \in \ALPHABET S.
\end{align*}
Thus, $\pi_v$ satisfies \eqref{eq:kappa}.

Next, from the definition of the Bellman operator, we have    
\begin{align*}
c_{\pi_v}(s) = 
    & [ \BELLMAN^{\pi_v} v](s) - {\gamma}
    \int_{\ALPHABET A} \pi_{v}(da \mid s)
    \int_{\ALPHABET S} v(s') P(ds'\mid s,a)
    \notag\\
    \le &
    [ \BELLMAN^{\pi_v} v](s) + {\gamma} |c_{\min}/(1-\gamma)|,
\end{align*}
where the inequality above is due to the fact that $ - v(s) \le  |c_{\min}/(1-\gamma)|$ for all $v \in \ALPHABET V_w$.
Therefore,
\begin{align*}
\| c_{\pi_v}   \|_w 
= & \sup_{s \in \ALPHABET S}
    \frac{ |c_{\pi_v}(s)| }{ w(s) }
\notag\\
\displaybreak[1]
\le & \sup_{s \in \ALPHABET S}
    \frac{[ \BELLMAN^{\pi_v} v](s) + {\gamma}|c_{\min}/(1-\gamma)|}{w(s)}
\notag\\
\le & \| \BELLMAN^{\pi_v} v   \|_w  +{\gamma} |c_{\min}/(1-\gamma)| \notag \\
= & \| \BELLMAN^{\star} v   \|_w  +{\gamma} |c_{\min}/(1-\gamma)| 
< \infty .
\end{align*}
Thus, $\pi_v$ satisfies \eqref{eq:c-pi}. 
Hence, $\pi_v \in \Pi_S(\tuple)$.

With $v=\hat V^\star$ in the above argument, it follows that $\pi_{\hat V^\star}  \in \Pi_S(\tuple)$. Hence, Assumption~\ref{ass:N1} is satisfied.

Similarly, for any $v \in \ALPHABET V_w$, let $\hat\pi_v$ denote a policy in $\hat{\ALPHABET G}(v)$. Using arguments identical to the ones used above, we can show that $\hat\pi_v \in \hat\Pi_S(\tuple)$. Setting $v= V^\star$ then implies that 
$\hat \pi_{ V^\star}  \in \hat \Pi_S(\tuple)$. Hence, Assumption~\ref{ass:N2} is satisfied.

\section{Proof of Lemma~\ref{lem:mismatch}}\label{app:mismatch}

By definition, $\MISMATCH^{\max} v = \sup_{a \in  \ALPHABET A} \MISMATCH^{\pi_a} v$. Since, by Assumption \ref{assm:openloop}, $\pi_a \in \Pi_S(\tuple) \cap \hat \Pi_S(\tuple)$ for all $a \in \ALPHABET A$, it follows that 
\begin{equation}\label{eq:lem-mismatch-pt-3-step-1}
    \MISMATCH^{\max} v \leq \sup_{\pi \in \Pi_S(\tuple) \cap \hat \Pi_S(\tuple)} \MISMATCH^\pi v.
\end{equation}

Define
\begin{multline*}
    \Xi^{(s,a)}v = 
     c(s,a) - \hat c(s,a) \\
     + \gamma \int_{\ALPHABET S} v(s') P(ds' | s,a) - \gamma \int_{\ALPHABET S} v(s')\hat P(ds'|s,a).
\end{multline*}
Then, we have that 
\[
 \MISMATCH^{\max}v = 
 \sup_{a \in \ALPHABET A} \| \BELLMAN^{\pi_a} v - \hat {\BELLMAN}^{\pi_a} v \|_w 
 =\sup_{a \in  \ALPHABET A} \sup_{s \in \ALPHABET S } \frac{\bigl| \Xi^{(s,a)}v \bigr|}{w(s)}
\]
and for every $\pi \in \Pi_S(\tuple) \cap \hat \Pi_S(\tuple)$, 
\[
    [\BELLMAN^{\pi} v](s) - [\hat \BELLMAN^{\pi} v](s) = 
    \int_{\ALPHABET A} \pi(da \mid s) \Xi^{(s,a)} v
\]
Therefore,
\begin{align}
    \MISMATCH^{\pi} v &= 
    \sup_{s \in \ALPHABET S} 
    \frac{\bigl| \int_{\ALPHABET A} \pi(da \mid s) \Xi^{(s,a)} v \bigr|}{w(s)}
    \notag \\
    &\le 
    \sup_{s \in \ALPHABET S } \sup_{a \in  \ALPHABET A}
    \frac{|\Xi^{(s,a)} v| }{w(s)}
    = \MISMATCH^{\max} v.
    \label{eq:lem-mismatch-pt-3-step-2}
\end{align}
Combining~\eqref{eq:lem-mismatch-pt-3-step-1} and~\eqref{eq:lem-mismatch-pt-3-step-2}, we get the first part of~\eqref{eq:mismatch-reln}

For the second part, note that for any set $\ALPHABET X$, $| \inf_{x \in \ALPHABET X} f(x) - \inf_{x \in \ALPHABET X} g(x)| \le \sup_{x \in \ALPHABET X}| f(x) - g(x) |$. Therefore, 
\[
    |[\BELLMAN^{\star} v](s) - [\hat \BELLMAN^{\star} v](s)| \le
    \sup_{a \in \ALPHABET A} |\Xi^{(s,a)} v|.
\]
Using the above inequality in the definition of $\MISMATCH^{\star}$, we get
\begin{align}
    \MISMATCH^{\star} v &= 
    \sup_{s \in \ALPHABET S} 
    \frac{ \bigl| [\BELLMAN^{\star} v](s) - [\hat \BELLMAN^{\star} v](s) \bigr|}{w(s)} \notag \\
    &\le    \sup_{s \in \ALPHABET S} 
    \frac{\sup_{a \in \ALPHABET A} \bigl| \Xi^{(s,a)} v \bigr|}{w(s)}
    = \MISMATCH^{\max} v,
    \label{eq:lem-mismatch-pt-3-step-3}
\end{align}
which establishes the second part of~\eqref{eq:mismatch-reln}.

\section{Proof of Theorem~\ref{thm:alpha_beta}}\label{app:alpha_beta}
\subsection{Proof of part~1}
The bound in terms of $\hat V^\star$ follows from Theorem~\ref{thm:bound}, part~1 (the bound in terms of $\hat V^\star$) and Lemma~\ref{lem:alpha-beta}, part~3.

For the bound in terms of $V^\star$, we can use Lemma~\ref{lem:alpha-beta}, part~3 and Theorem~\ref{thm:bound}, part~1 to write 
\begin{align}
    \hskip 1em & \hskip -1em
      \bigl\| V^{\hat \pi^\star} - V^\star \bigr\|_w
      =
      \frac{1} {\alpha_1}\bigl\| V^{\hat \pi^\star}_{\boldsymbol{\alpha}} - V^\star_{\boldsymbol{\alpha}} \bigr\|_w
      \notag \\
      \displaybreak[1]
      &
      \le
      \frac{1}{\alpha_1(1 - \gamma \kappa) }
      \MISMATCH^{\hat \pi^\star}_{\boldsymbol{\alpha}} V^\star_{\boldsymbol{\alpha}}
      +
    \frac{(1 + \gamma \kappa)}{\alpha_1(1 - \gamma \kappa)^2 }
      \MISMATCH^{\pi^\star, \hat \pi^\star} V_{\boldsymbol{\alpha}}^\star
       \notag \\
      &=
      \frac{1}{\alpha_1(1 - \gamma \kappa) }
      \MISMATCH^{\hat \pi^\star}_{\boldsymbol{\alpha}} (\alpha_1 V^\star)
      +
      \frac{(1 + \gamma \kappa)}{\alpha_1(1 - \gamma \kappa)^2 }
     \MISMATCH^{\pi^\star, \hat \pi^\star}  (\alpha_1 V^\star),
\end{align}
 where, in the last step, we used the fact that part 2 of Lemma \ref{lem:alpha-beta} implies that
\[
     \MISMATCH^{\hat \pi^\star}_{\boldsymbol{\alpha}} V^\star_{\boldsymbol{\alpha}} = \MISMATCH^{\hat \pi^\star}_{\boldsymbol{\alpha}} (\alpha_1 V^\star)  
     \quad \text{and}\quad
     \MISMATCH_{\boldsymbol{\alpha}}^{\pi^\star, \hat \pi^\star}  V^\star_{\boldsymbol{\alpha}} = \MISMATCH_{\boldsymbol{\alpha}}^{\pi^\star, \hat \pi^\star}  (\alpha_1 V^\star).
\]
\subsection{Proof of part~2}
Assumption~\ref{ass:N1_alt} is the same as Assumption~\ref{ass:N1} for model $\ALPHABET M_{\boldsymbol{\alpha}}$. Hence, the result follows immediately from Theorem~\ref{thm:bound}, part~2 and Lemma~\ref{lem:alpha-beta}, part~3. 
\subsection{Proof of part~3}
Assumption~\ref{ass:N2_alt} is the same as Assumption~\ref{ass:N2} for model $\ALPHABET M_{\boldsymbol{\alpha}}$. Hence, the result follows from Theorem~\ref{thm:bound}, part~3 together with Lemma~\ref{lem:alpha-beta}, part~3 and the fact that part~2 of Lemma~\ref{lem:alpha-beta} implies
\[
     \MISMATCH^{\star}_{\boldsymbol{\alpha}} V^\star_{\boldsymbol{\alpha}} = \MISMATCH^{\star}_{\boldsymbol{\alpha}} (\alpha_1 V^\star).
\]

\section{Proof of Proposition~\ref{prop:LQR}}\label{append:lqr}
When $\alpha_1 = 1$, Assumptions~\ref{ass:N1} and~\ref{ass:N1_alt} are equivalent. To apply Theorem~\ref{thm:alpha_beta}, we consider $\alpha_1=1$ and an arbitrary $\alpha_2$. Then the corresponding Bellman updates for $\hat{V}^\star(s)$ can be calculated as
\begin{align*}
    &
    \BELLMAN^{\star}_{(1, \alpha_2)}\hat{V}^\star(s)
    \notag\\
    &=  
    s^\intercal \left(
    Q + \gamma A^\intercal \hat P A - \gamma^2 A^\intercal \hat PB(R + \gamma B^\intercal \hat P B)^{-1} B^\intercal \hat P A
    \right)
    s
    \notag\\
    & \quad + \gamma (\hat q + \Tr(\Sigma_W \hat P)) 
    + \alpha_2,
\end{align*}
and 
\begin{align*}
    \hat \BELLMAN^\star\hat{V}^\star(s) &= \hat V^\star (s) = s^\intercal \hat P s + \hat q.
    \notag \\
    &= s^\intercal \hat P s + \gamma(\hat q + \Tr(\hat \Sigma_W \hat P))
\end{align*}
where the last term uses the fact that $\hat q = \gamma \Tr(\hat \Sigma_W \hat P)/(1-\gamma)$.

Therefore, we have
\begin{align*}
    &|\BELLMAN^{\star}_{(1, \alpha_2)}\hat{V}^\star(s)
     - \hat{\BELLMAN}^{\star}\hat{V}^\star(s)|
    \\
    & \quad = 
    \left|
    s ^\intercal D^\star
    s
    + \gamma \Tr((\Sigma_W -\hat \Sigma_W) \hat P)
    + \alpha_2
    \right|,
\end{align*}
where $D^\star$ is given by \eqref{eq:Dstar}.

Note that $\hat \pi^\star(s) = - \hat K^\star s = -\gamma(\hat R + \gamma \hat B^\intercal \hat P \hat B)^{-1}\hat B^\intercal \hat P \hat A s$. As a result, 
for any $\alpha_2$, $\BELLMAN^{\hat \pi^\star}_{(1, \alpha_2)}\hat{V}^\star(s)$ is given by
\begin{align*}
    \BELLMAN^{\hat \pi^\star}_{(1, \alpha_2)}\hat{V}^\star(s) 
    &= 
    s^\intercal
    \big(
    Q +  (\hat K^\star)^\intercal R \hat K^\star + \gamma A_{\hat K^\star}^\intercal \hat P A_{\hat K^\star}
    \big) 
    s  
    \notag\\
    & \quad + \gamma (\hat q + \Tr(\Sigma_W \hat P))
    + \alpha_2
    ,
\end{align*}

and $\hat \BELLMAN^{\hat \pi^\star}\hat{V}^\star(s) = \hat{\BELLMAN}^{\star}\hat{V}^\star(s) = \hat V^\star (s)$.
Therefore, we have
\begin{align*}
    &|
    \BELLMAN^{\hat \pi^\star}_{(1, \alpha_2)}\hat{V}^\star(s) 
     - \hat{\BELLMAN}^{\hat \pi^\star}\hat{V}^\star(s)|
    \\
    &\quad = 
    \left|
    s ^\intercal D^{\hat \pi^\star}
    s
    + \gamma \Tr((\Sigma_W -\hat \Sigma_W) \hat P)
    + \alpha_2
    \right|,
\end{align*}
where $D^{\hat \pi^\star}$ is given by \eqref{eq:Dpihatstar}. Then, the Bellman mismatches functionals of Section~\ref{sec:alpha_beta} for $\hat V$ with $(1, \alpha_2)$ can be calculated as follows:
\begin{align*}
    \MISMATCH^{\hat \pi^\star}_{(1, \alpha_2)}\hat{V}^\star
     &= \sup_{s \in \ALPHABET S} \frac{\left|
    s ^\intercal D^{\hat \pi^\star}
    s
    + d_\Sigma + \alpha_2
    \right|}{w(s)}
    ,
    \\
    \MISMATCH^{\star}_{(1, \alpha_2)}\hat{V}^\star
     &= \sup_{s \in \ALPHABET S} \frac{\left|
    s ^\intercal D^\star
    s
    + d_\Sigma + \alpha_2
    \right|}{w(s)},
\end{align*}
where $d_\Sigma$ is given by \eqref{eq:dsigma}.

Eq.~\eqref{eq:LQ-bound} then then follows from Theorem~\ref{thm:alpha_beta} part 2 by observing that for any symmetric matrix $D$
\[
\sup_{s \in \ALPHABET S} \frac{\left|
    s ^\intercal D s
    + d_\Sigma + \alpha_2
    \right|}{1 + \ell s^\intercal s}
    \le
    \max\bigg\{ \frac{\rho(D)}{\ell} ,
    | d_\Sigma + \alpha_2 | \biggr\}.
\]

\section{Proof of Lemma~\ref{lem:mismatch_bounds}}\label{app:mismatch_bounds}
\subsection{Proof of part 1}
For any $\scale$ with $\alpha_2 > 0$, we have 
\begin{align}\label{eq:IPM_pf2}
    \MISMATCH^{\pi,\hat\pi}_{\boldsymbol{\alpha}} v =& \sup_{s \in \ALPHABET S}\frac{\bigl| \BELLMAN_{\boldsymbol{\alpha}}^{\pi} v(s) -
    \hat{\BELLMAN}^{\hat\pi}v(s) \bigr| }{w(s)} \notag \\
    \leq& \sup_{s \in \ALPHABET S} \frac{\bigl|\alpha_1 c_{\pi}(s)+\alpha_2 - \hat c_{\hat \pi}(s) \bigr|}{w(s)} \nonumber \\
    \quad& + \gamma \sup_{s \in \ALPHABET S}\frac{\biggl|\int_{\ALPHABET S} v(s')\Bigl[ P_{\pi}\bigl(ds'|s)\bigr) - \hat P_{\hat\pi}\bigl(ds'|s) \bigr)\Bigr] \biggr|}{w(s)} 
    \displaybreak[1]
    \notag \\
    \leq& \varepsilon_{\boldsymbol{\alpha}}(\pi,\hat\pi)  \notag \\
     \quad& + \gamma \rho_{\F}(v)\sup_{s \in \ALPHABET S} \frac{d_{\F}\Bigl(P_{\pi}\bigl(\cdot|s\bigr),\hat P_{\hat\pi}\bigl(\cdot|s\bigr)\Bigr)}{w(s)} \nonumber \\
     =& \varepsilon_{\boldsymbol{\alpha}}(\pi,\hat\pi) + \gamma \rho_{\F}(v) \delta_{\F}(\pi,\hat\pi)
\end{align}
\subsection{Proofs of part~2 and~3}
Part~2 follows because $\MISMATCH^\star_{\boldsymbol{\alpha}} \hat V^\star = \MISMATCH^{\mu^\star,\hat \pi^\star}\hat V^\star$. Similarly, part~3 follows because $\MISMATCH^\star_{\boldsymbol{\alpha}} (\alpha_1 V^\star) = \MISMATCH^{\pi^\star,\hat \mu^\star} (\alpha_1 V^\star)$.
\subsection{Proof of part 4}
Define
\begin{multline*}
    \Xi^{(s,a)}_{\boldsymbol{\alpha}}v = 
    \alpha_1 c(s,a) + \alpha_2 - \hat c(s,a) \\
     + \gamma \int_{\ALPHABET S} v(s') P(ds' | s,a) - \gamma \int_{\ALPHABET S} v(s')\hat P(ds'|s,a).
\end{multline*}

Then, we have
\begin{align}\label{eq:IPM_pf3}
         \MISMATCH_{\boldsymbol{\alpha}}^{\max}v &=
         \sup_{a \in  \ALPHABET A} \| \BELLMAN^{\pi_a}_{\boldsymbol{\alpha}}v - \hat{\BELLMAN}^{\pi_a}v \|_w \notag \\
         &=\sup_{(s,a) \in \ALPHABET S \times \ALPHABET A}\frac{\bigl| \Xi^{(s,a)}_{\boldsymbol{\alpha}}v \bigr|}{w(s)}
         \notag \\
    \displaybreak[1]
&\leq \sup_{(s,a) \in \ALPHABET S \times \ALPHABET A} \frac{\bigl|\alpha_1 c(s,a) +\alpha_2 - \hat c(s,a)\bigr|}{w(s)} \nonumber \\
&\quad + \gamma \sup_{(s,a) \in \ALPHABET S \times \ALPHABET A}\frac{\bigl| \int_{\ALPHABET S} v(s')[ P(ds' | s,a) - \hat P(ds'|s,a)]\bigr|}{w(s)} \nonumber \\
    \displaybreak[1]
&\leq \varepsilon_{\boldsymbol{\alpha}}^{\max} + \gamma \rho_{\F}(v)\sup_{(s,a) \in \ALPHABET S \times \ALPHABET A}\frac{d_{\F}\bigl(P(\cdot|s,a),\hat P(\cdot|s,a)\bigr)}{w(s)} \notag \\
&= \varepsilon_{\boldsymbol{\alpha}}^{\max} + \gamma \rho_{ \mathfrak F}(v)\delta_{\mathfrak F}^{\max}.
\end{align}

\begin{IEEEbiography}{Berk Bozkurt} (Student Member, IEEE) received his BSc degree in Electrical and Electronics Engineering from Bilkent University, Ankara, Turkey, in 2021. He is currently a MSc student in Electrical and Computer Engineering Department at McGill University, Montreal, Canada. His research interests include reinforcement learning, game theory, stochastic control and Markov decision theory.

\end{IEEEbiography}
\begin{IEEEbiography}{Aditya Mahajan} (Senior Member, IEEE) is Professor of Electrical and Computer Engineering at McGill University, Montreal, Canada. He received the B.Tech degree in Electrical Engineering from the Indian Institute of Technology, Kanpur, India  in 2003 and the MS and PhD degrees in Electrical Engineering and Computer Science from the University of Michigan, Ann Arbor, USA in 2006 and 2008, respectively.
He serves as Associate Editor of  Transactions on Automatic Control, Control Systems Letters, and Math. of Control, Signal, and Systems.  
\end{IEEEbiography}

\begin{IEEEbiography}{Ashutosh Nayyar} 
 (Senior Member, IEEE) is an Associate Professor of Electrical and Computer Engineering at the University of Southern California. He  received a M.S. degree in electrical engineering and computer science, a M.S. degree
in applied mathematics, and a Ph.D. degree
in electrical engineering and computer science
from the University of Michigan, Ann Arbor, MI,
USA, in 2008, 2011, and 2011, respectively.
 His research interests are in decentralized stochastic control,
decentralized decision-making in sensing and communication systems,
reinforcement learning, game theory and  mechanism design.
\end{IEEEbiography}

\begin{IEEEbiography}{Yi Ouyang} received the B.S. degree
in Electrical Engineering from the National Taiwan
University, Taipei, Taiwan in 2009, and the M.Sc
and Ph.D. in Electrical Engineering and Computer Science at the University
of Michigan, in 2012 and 2015, respectively. He is
currently a researcher at Preferred Networks, Burlingame, CA. His research interests include reinforcement
learning, stochastic control, and stochastic dynamic
games.
\end{IEEEbiography}

\vfill

\end{document}